\documentclass[12pt]{article}

\usepackage{amsfonts}
\usepackage{amscd}
\usepackage{amsmath}
\usepackage{epsfig}
\usepackage{amsthm}
\usepackage{amssymb}
\usepackage{amsbsy}
\usepackage{latexsym}
\usepackage{eucal}
\usepackage{mathrsfs,bm}
\usepackage{mathtools}
\usepackage{tikz}
\usepackage[colorlinks]{hyperref}
\usepackage{xcolor}
\usepackage{soul}

 \newtheorem{theorem}{Theorem}[section]

\newtheorem{proposition}[theorem]{Proposition}
\newtheorem{corollary}[theorem]{Corollary}
\newtheorem{remark}[theorem]{Remark}
\newtheorem{lemma}[theorem]{Lemma}

\def\be{\begin{equation}}
\def\ee{\end{equation}}
\def\ben{\begin{displaymath}}
\def\een{\end{displaymath}}
\def\baa{\begin{eqnarray}}
\def\eaa{\end{eqnarray}}

\def\ba{\begin{array}}
\def\ea{\end{array}}
\renewcommand{\leq}{\leqslant}
\renewcommand{\geq}{\geqslant}

\newcommand{\supp}{\operatorname{supp}}

\newcommand{\Tr}{\operatorname{Tr}}

\makeatletter
\@addtoreset{equation}{section}
\makeatother

\begin{document}
\title{Determinants of Laplacians for constant curvature metrics  with three conical singularities on $2$-sphere}

\author{Victor Kalvin}

\date{}
\maketitle

\begin{abstract} 
We deduce an explicit closed formula for the  zeta-regularized spectral determinant of the Friedrichs Laplacian on the Riemann sphere equipped with arbitrary constant curvature (flat, spherical, or hyperbolic) metric having three conical singularities of order $\beta_j\in(-1,0)$ (or, equivalently,  of angle $2\pi(\beta_j+1)$).  
We  show that  among the metrics with a fixed value of the sum $\beta_1+\beta_2+\beta_3$ and a fixed surface area, those with  $\beta_1=\beta_2=\beta_3$ correspond to  a stationary point of the determinant. If, in addition,  the surface area is sufficiently small, then the stationary point is a minimum. 

As a crucial step towards obtaining these results we find a relation between the determinant of Laplacian and the Liouville action introduced by A. Zamolodchikov and Al. Zamolodchikov in connection with the celebrated DOZZ formula for the three-point  structure constants of the Liouville field theory. 
\end{abstract}

\section{Introduction and main results}\label{Intro}
\subsection{Introduction}
There are not so many geometric settings  in which explicit closed formulae for the determinants of Laplacians are known, see e.g.~\cite{AKR,Au-Sal,AS2,KalvinCCM,Klevtsov,Polch,Spreafico,SpreaficoZerbini,Weisberger}. 
In this paper we derive one for the determinant of Friedrichs Laplacian  on the Riemann sphere equipped with arbitrary constant curvature (flat, spherical, or hyperbolic) metric having three conical singularities.  
We explicitly express the determinant in terms of the surface area and the orders of  conical singularities (angles).  This generalizes previously known explicit formulae for the determinant of Friedrichs Laplacians on the isosceles Euclidean triangle envelopes~\cite{AS2,KalvinJGA},  on the spindles~\cite{KalvinJFA,Klevtsov,SpreaficoZerbini}, and on the standard round spheres, e.g.~\cite{OPS}. Besides,  this allows one  to  independently deduce explicit expressions for the determinant of  Friedrichs  Dirichlet Laplacians on the constant curvature cones  (and, in particular, the one for the flat cones~\cite{Spreafico}) as in~\cite{KalvinCCM}.

 We also study extremal properties of the determinant considered as a function of the orders  $\beta_j\in(-1,0)$ of conical singularities (or, equivalently, as a function of the angles $2\pi(\beta_j+1)$), while the sum  $|\pmb\beta|:=\beta_1+\beta_2+\beta_3$  remains fixed. We show that   $\beta_1=\beta_2=\beta_3=|\pmb\beta|/3$ is always a stationary point (for any fixed $|\pmb\beta|\in(-3,0)$),  it  correspond to the most symmetrical surfaces.  Moreover, the stationary point is a minimum of the determinant provided that among the surfaces with a fixed value of $|\pmb\beta|$ we consider only those with surface area $S$ not exceeding a certain value that may depend on $|\pmb\beta|$.

The problem we study  is closely related to the celebrated DOZZ formula of H. Dorn, H.-J. Otto~\cite{DO} and   A. Zamolodchikov, Al. Zamolodchikov~\cite{Z-Z}. The DOZZ formula is a heuristically deduced explicit expression  for the three-point structure constant of the Liouville conformal field theory. In the classical limit  (i.e. as the Liouville coupling constant goes to zero) the behaviour of the  structure constant is governed by the Liouville action evaluated on the potential $\phi$ of a constant curvature metric $e^{2\phi}|dz|^2$  on the Riemann sphere $\overline{\Bbb C}=\Bbb C\cup\{\infty\}$ with conical singularities of order $\beta_j$ at three distinct points $p_j\in\Bbb C$~\cite[Eq. (3.20)]{Z-Z}; see also Remark~\ref{GLA}.

 We study  the determinant of (the Friedrichs selfadjoint extension of) the Laplacian $-4e^{\phi(z)}\partial_z\partial_{\bar z}$ induced by the metric $e^{2\phi}|dz|^2$ on  $\overline{\Bbb C}$. As a crucial step in our study of the determinant, we express it in terms of the  Liouville action 
and some other explicit functions of  $\beta_j$ and $S$, where $S$ is the area of the surface $(\overline{\Bbb C}, e^{2\phi}|dz|^2)$. On this step we rely on the Polyakov-Alvarez type formula and the BFK (Burghelea-Friedlander-Kappeler~\cite{BFK}) formula proven in~\cite{KalvinJFA}.  

 The resulting anomaly formula for the determinant  of Laplacian is not exactly of Polyakov or Polyakov-Alvarez type: 
It is a new formula that relates the determinant of the Laplacian on the singular constant curvature surface $(\overline{\Bbb C}, e^{2\phi}|dz|^2)$  to the determinant of the Euclidean Laplacian on a disk of radius $R$ as  $R\to\infty$. This  allows us to make a link between our study of  determinants of Laplacians and the  work of  A. Zamolodchikov and  Al. Zamolodchikov on conformal bootstrap in the Liouville field theory~\cite{Z-Z}, where the Liouville action is introduced on a disk of radius $R$ as $R\to\infty$, and the Euclidean metric is used as a background metric on the disk.

We consider the constant curvature metrics of fixed area $S$, or, without loss of generality,  the metrics of unit area (as the case of area $S\neq 1$ can be immediately reduced to the case $S=1$ with the help of the standard  rescaling property of the zeta-regularized determinants). We show that the Liouville action, evaluated on the potential $\phi$ of  such a metric,   satisfies a system of  governing differential equations.
 This system is slightly different form the one in~\cite{Z-Z}, where the curvature of the metrics (instead of the surface area) is fixed, but can also be easily integrated. 
 As a result, we obtain an explicit closed expression for the 
 Liouville action in terms of the orders $\beta_j$ of conical singularities. The expression is new, it involves derivatives of Hurwitz and Riemann zeta functions. This together with the anomaly formula leads to an explicit closed formula for the determinant of Laplacian in terms of $\beta_j$.
 
This result can be of interest in Number theory (e.g. in connection with the problem of finding particular values of the Selberg zeta function) and various areas of theoretical physics.
 We believe that it will also allow us to explicitly evaluate the determinant of Laplacians on the smooth and singular surfaces obtained by cutting and gluing  genus zero constant curvature  surfaces with three conical singularities;  for some results in this direction see~\cite{KalvinJGA19,IMRN,Bul}.

In the flat case our explicit formula for the determinant significantly simplifies and, in particular,  generalizes the results in~\cite{KalvinJGA}, where only the  case  $\beta_1=\beta_3=:\beta$ and $\beta_2=-2-2\beta$ is studied. It is also interesting to note that the celebrated partially heuristic Aurell-Salomonson formula for determinants of Laplacians on polyhedra with spherical topology~\cite[eqn. (50)]{AS2}  returns an equivalent result; the Aurell-Salomonson formula received a rigorous mathematical proof in~\cite[Sec. 3.2]{KalvinJFA}.

In the limit cases of a spindle (i.e. when $\beta_j\to 0^-$ while $\beta_k=\beta_\ell$ with $k,\ell\neq j$) and of a standard round sphere (i.e. when $\beta_j\to 0^-$ for $j=1,2,3$) we independently recover the corresponding explicit expressions for the determinant  known from~\cite{KalvinJFA,Klevtsov,SpreaficoZerbini}.

 Let us also mention that the anomaly formula for the determinant of Laplacian can be easily extended to the case of a constant curvature metric with any number $n$ of conical singularities on a $2$-sphere (the proof remains essentially the same). However,  an explicit construction of the uniformazation map for the $2$-sphere with $n$ conical singularities is needed in order to turn the anomaly formula into a closed explicit formula for the determinant of Laplacian. As is well-known, starting from $n=4$, an explicit construction of the uniformazation map is an open long standing problem, closely related to the problem of finding  explicit expressions for the so-called conformal blocks, accessory parameters, and the Liouville action, see e.g.~\cite{CMS,HJ,T-Z,Z-Z}. (In the case of $n=3$  the construction is essentially well-known, starting from $n=4$ it is known for some particular ``symmetrical'' cases only, but not in general.)  This represents a substantial  difficulty towards obtaining an explicit formula for the determinant of Laplacian in the case of $n\geq 4$.

Having at hands the anomaly formula for the determinant of Laplacian and the governing differential equations for the Liouville action, it is fairly  easy to study  extremal properties  of the determinant as a function of the orders $\beta_j$ of conical singularities   (angles).    We derive formulae for the second order derivatives of the determinant and conclude that for each $|\pmb\beta|\in(-3,0)$ the stationary point $\beta_1=\beta_2=\beta_3=\frac {|\pmb\beta|} 3$  is a minimum of the determinant if the surface area $S$ is sufficiently small, i.e. $S<S_0=S_0(|\pmb\beta|)$. 

Recall that for smooth varying metrics,  extrema of  determinants of Laplacians and compactness of families of isospectral metrics were studied  in a series of papers~\cite{OPS,OPS.5,OPS1} by Osgood, Phillips, and Sarnak; see also~\cite{Sarnak} for a review and~\cite{Kim} for an extension of their results.   
While only a few results on extremal properties of determinants under variation of angles (or, equivalently, orders of conical singularities) are available yet: for the isosceles Euclidean triangle envelopes~\cite{KalvinJGA}, and for the spherical metrics with two conical singularities on a $2$-sphere~\cite[Sec. 3.1]{KalvinJFA}.

 This paper is organized as follows. The next Subsection~\ref{MainResults}  contains preliminaries and  the main results: Anomaly formula for the determinant of Laplacian (Theorem~\ref{THMdet}), Explicit formula for the Liouville action (Theorem~\ref{LAexplInt}), Resulting explicit formula for the determinant of Laplacian (Corollary~\ref{COR}), and Stationary points of the determinant (Theorem~\ref{TSP}).

    In Section~\ref{DetLap} we prove the anomaly formula (Theorem~\ref{THMdet}). 
    
    In Section~\ref{LiAc} we  study the Liouville action. Namely, in Subsection~\ref{SecGE} we deduce  the system of governing differential equations for the Liouville action. In Subsection~\ref{LA_S2} we  integrate the system of governing equations and find the constant of integration, this proves Theorem~\ref{LAexplInt} and Corollary~\ref{COR}.

    In Section~\ref{FSL} we specify the explicit formula for the determinant in the case of a flat and limit spherical metrics.  In particular, in Subsection~\ref{FLAT_METRICS} we   show that in the case of a flat metric the explicit formula for the determinant in Corollary~\ref{COR} significantly simplifies and generalizes our previous results in~\cite{KalvinJGA}. In Subsections~\ref{SubSpindle} and~\ref{StRoundSph} we show that  in the limit cases of a spherical metric our explicit formula for the determinant correctly reproduces the corresponding results in~\cite{KalvinJFA,Klevtsov,SpreaficoZerbini}.
 
 Finally, in Section~\ref{SecSPoints} we study stationary points of the determinant. We deduce explicit formulas for the (logarithmic) second order derivatives of the determinant, and prove Theorem~\ref{TSP}.
 
 Auxiliary results on conformal constant curvature singular metrics on the Riemann sphere are collected in Appendix~\ref{NP}. 
In particular, in  Subsection~\ref{EandU}  we discuss conditions that guarantee existence and uniqueness of a constant curvature conformal metric with three conical singularities. Then in Subsection~\ref{UACCM} we explicitly construct  the unit area metrics as functions of  the orders of conical singularities.

 \subsection{Preliminaries and main results}\label{MainResults}
Consider the Riemann sphere $\overline{\Bbb C}=\Bbb C\cup\{\infty\}$ with three distinct marked points $p_j\in\Bbb C$.
By applying a suitable M\"obius transformation we  can always normalize so that $p_1$, $p_2$, and $p_3$ are any three distinct points of our choice. 
 Let  $\beta_1$, $\beta_2$, and $\beta_3$ be  three real numbers such that $-1<\beta_j<0$ and   
 $\beta_j-{|\pmb\beta|}/2>0$. Here $$|\pmb\beta|=\beta_1+\beta_2+\beta_3$$ stands for the degree of the divisor $\pmb\beta=\sum_{j=1}^3\beta_j\cdot p_j$. The divisor is a formal sum.
 
  It turns out that for any $S>0$ there exists a unique conformal  constant curvature metric of area $S$  representing the divisor $\pmb\beta$. This means that the metric on the Riemann sphere $\overline{\Bbb C}$ has conical singularities of order $\beta_j$  (or, equivalently,  of angle $2\pi(\beta_j+1)$) at the points $p_j$. Or, more precisely,  that the metric can be written in the form $S\cdot e^{2\phi}|dz|^2$, where $\phi$ is a smooth function on $\Bbb C\setminus\{p_1,p_2,p_3\}$ satisfying the Liouville equation 
\begin{equation}\label{LEq1}
e^{-2\phi}(-4\partial_z\partial_{\bar z}\phi)=2\pi(|\pmb\beta|+2),\quad z\in\Bbb C\setminus\{p_1,p_2,p_3\},
\end{equation}
and having the following asymptotics at vicinities of the points $p_j\in \Bbb C$ and at infinity:
\be\label{ASphi}
\begin{aligned}
\phi(z) & =\beta_j\log|z-p_j|+\phi_j+o(1), \quad z\to p_j,
\\
 \phi(z)& =-2\log|z| +\phi_\infty+o(1), \quad z\to \infty.
\end{aligned}
\ee
Here the coefficients $\phi_j$  and $\phi_\infty$ are some functions of the orders $\beta_j$ of conical singularities. 

In what follows it is  important that  the coefficients $\phi_j$  can be explicitly expressed in terms of $\beta_j$. In fact,  (up to a normalization of the marked points $p_j$) the  unit area metric $e^{2\phi}|dz|^2$ can be explicitly constructed as the pullback  of the model metric $$
\ {4(1+2\pi(|\pmb\beta|+2)|w|^2)^{-2}|d w|^2}
$$
 by the Schwarz triangle function. Details can be found in Appendix~\ref{NP}, where, in particular, 
we find  explicit expressions for  the coefficients $\phi_j$ that correspond to the normalization $p_1=-1$, $p_2=0$ and $p_3=1$ (see Eq.~\eqref{phi123} and Eq.~\eqref{Phi}).
As for the coefficient $\phi_\infty$,  we only need to know that it is well defined.

By the Gauss-Bonnet theorem~\cite{Troyanov} the (regularized) Gaussian curvature $K$ of the metric $S\cdot e^{2\phi}|dz|^2$ is given by $K= {2\pi(|\pmb\beta|+2)} /S$.
Note that in the spherical ($K>0$) case the condition  $\beta_j-{|\pmb\beta|}/2>0$ is necessary and sufficient for the existence of a spherical metric with three conical singularities of order $\beta_j\in(-1,0)$, e.g.~\cite{Eremenko,LT,Troyanov,UY}. While 
 in the hyperbolic ($K<0$) and flat ($K=0$)  cases  this condition is a priori satisfied.

Consider a (unique) unit area constant curvature metric $e^{2\phi}|dz|^2$ representing the divisor $\pmb\beta$.  Let $L_{\pmb\beta}^2$  stand for  the space of functions on $\overline{\Bbb C}$ with finite norms
$$
\|f\|_{\pmb\beta}=\left(\int_{\Bbb C} |f(z)|^2e^{2\phi(z)}\frac {dz\wedge d\bar z}{-2i}\right)^{1/2}.
$$
In particular, the equality $\|1\|_{\pmb\beta}=1$ reflects the fact that   $e^{2\phi}|dz|^2$ is a unit area metric. 

 The Laplacian  $\Delta_{\pmb\beta}=-e^{-2\phi}4\partial_z\partial_{\bar z}$  on the Riemann sphere $\overline{\Bbb C}$ 
is an unbounded operator in the Hilbert space $L_{\pmb\beta}^2$, initially defined on the smooth functions supported outside of the conical singularities at $z=p_j$. The operator $\Delta_{\pmb\beta}$ is densely defined, but not essentially selfadjoint. We pick the Friedrichs selfadjoint extension, which we still denote by $\Delta_{\pmb\beta}$ and call the Friedrichs Laplacian or simply Laplacian for short.  The spectrum of $\Delta_{\pmb\beta}$ consists of non-negative isolated eigenvalues $0=\lambda_0<\lambda_1\leq \lambda_2\dots$ of finite multiplicity, and its  zeta-regularized spectral determinant can be introduced in the standard well-known way: 

 The spectral zeta function $s\mapsto\zeta_{\pmb\beta}(s)$ of $\Delta_{\pmb\beta}$,  defined by the equality $\zeta_{\pmb\beta}(s)=\sum_{j>0}\lambda_j^{-s}$ for $\Re s>1$,  extends by analyticity to a neighbourhood of $s=0$~\cite{KalvinJFA}. The zeta-regularized spectral determinant of the Laplacian $\Delta_{\pmb\beta}$ is given by 
 $$
 \det\Delta_{\pmb\beta}:=\exp(-\zeta_{\pmb\beta}'(0)). 
 $$
 This is a modified determinant, i.e. with zero eigenvalue excluded.
 
   Notice that  the  spheres $(\overline{\Bbb C}, e^{2\phi}|dz|^2)$ and  $(\overline{\Bbb C}, e^{2\phi\circ f}|\partial_z f|^2|dz|^2)$ are isometric via $f$, where $f: \overline{\Bbb C}\to \overline{\Bbb C}$ is a M\"obius transformation.  
 Thus $\det\Delta_{\pmb\beta}$ is invariant under the M\"obius transformations, and without loss of generality we can assume that the marked points $p_j$ are normalized, for instance,  so that $p_1=-1$, $p_2=0$, and $p_3=1$.

Now we are in position to formulate the first result of this paper: an anomaly formula for the determinant of Laplacian $\Delta_{\pmb\beta}$ that includes  the Liouville action as one of its terms. 

\begin{theorem}[Anomaly formula for determinant of Laplacain] \label{THMdet}  Assume that $\beta_j\in(-1,0)$ and $\beta_j-\frac{|\pmb\beta|}2>0$ for  $j=1,2,3$. 
Let  $\Delta_{\pmb\beta}$ stand for the Friedrichs Laplacian  on the Riemann sphere $\overline{\Bbb C}$ equipped with (unique)  \underline{unit area} constant  curvature   conformal metric $e^{2\phi}|dz|^2$ representing the divisor
\begin{equation}\label{DIVISOR}
\pmb\beta=\beta_1\cdot(-1)+\beta_2\cdot 0+\beta_3\cdot 1.
\end{equation}
Introduce the Liouville action  
\begin{equation}\label{LA1}
\mathcal S_{\pmb\beta}[\phi]:=   2\pi(|\pmb\beta|+2)\left(\int_{\Bbb C}\phi e^{2\phi}\frac {dz\wedge d\bar z }{-2i}-1\right)+2\pi\sum_{j=1}^3 \beta_j\phi_j+4\pi\phi_\infty
\end{equation}
and the functional  
\begin{equation}\label{FH}
\mathcal H_{\pmb\beta}[\phi]:=\exp\left\{2\sum_{j=1}^3 \left(\beta_j +1-\frac{1}{\beta_j+1}\right)\phi_j\right\},
\end{equation}
where the coefficients $\phi_j$ and $\phi_\infty$ are the same as in the asymptotics~\eqref{ASphi} (with $p_1=-1$, $p_2=0$, and $p_3=1$).

Then for the zeta-regularized spectral determinant of $\Delta_{\pmb\beta}$ we have 
\begin{equation}\label{DetDelta1}
\begin{aligned}
\log  {\det\Delta_{\pmb\beta}}=- \frac {|\pmb\beta|+1} 6-\frac 1 {12\pi}\Bigl(\mathcal S_{\pmb\beta}[\phi]-\pi \log \mathcal H_{\pmb\beta}[\phi]\Bigr)
-\sum_{j=1}^{3}\mathcal C(\beta_j)
\\
 -\frac 4 3 \log 2 -4\zeta_R'(-1) -\log\pi,
 \end{aligned}
 \end{equation}
where the function $\beta_j\mapsto \mathcal C(\beta_j)$ is  defined by the equality
\begin{equation}\label{Cb}
\mathcal C(\beta)=2\zeta'_B(0;\beta+1,1,1)-2\zeta_R'(-1)-\frac {\beta^2}{6(\beta+1)}\log 2 -\frac \beta {12}+\frac 1 2 \log(\beta+1).
\end{equation}
Here $\zeta'_B$ and $\zeta_R'$ stand for the derivatives with respect to $s$ of the Barnes double zeta function $\zeta_B(s;a,b,x)$ and the Riemann zeta function $\zeta_R(s)$ respectively.
\end{theorem}

As it was mentioned in the introduction, the equality~\eqref{DetDelta1}  is  a new anomaly formula for the determinant of Laplacian: It relates the determinant of the Laplacian $\Delta_{\pmb\beta}$ on the singular constant curvature sphere $(\overline{\Bbb C}, e^{2\phi}|dz|^2)$  to the determinant of Euclidean Laplacian on a disk of radius $R$ as  $R\to\infty$. The latter determinant does not appear in~\eqref{DetDelta1} because we  express it in terms of $R$  and then pass to the limit as $R\to \infty$;     for details we refer to  Section~\ref{DetLap}.

The definitions of the Liouville action $\mathcal S_{\pmb\beta}[\phi]$ in~\eqref{LA1},  the functional  $\mathcal H_{\pmb\beta}[\phi]$  in~\eqref{FH}, and the function $\beta\mapsto\mathcal C(\beta)$ in~\eqref{Cb} naturally come out of our study of the determinant of Laplacian. Later on we  show that $\mathcal S_{\pmb\beta}[\phi]$ is $4\pi$ times the regularized Liouville action introduced by A. Zamolodchikov and Al. Zamolodchikov in~\cite{Z-Z}. Our definition for  the Liouville action is also in agreement, for example, with those in~\cite{CMS,HJ,T-Z}. Let us also mention that there is a certain similarity between the functional $\mathcal H_{\pmb\beta}[\phi]$ and the K\"ahler potentials of the elliptic metrics in~\cite{T-Z2019}. 
A similar functional  also appears in~\cite{C-W, KalvinJFA}. 
The real analytic function $(-1,\infty)\ni\beta\mapsto \mathcal C(\beta)$, explicitly defined in~\eqref{Cb}, was first introduced in~\cite{KalvinJFA}. It describes the inputs into the determinant that come solely from the orders $\beta_j$ of conical singularities~\cite{KalvinJFA}, see also Remark~\ref{RCb}.

Recall that the coefficients $\phi_j$ can be explicitly expressed in terms of the orders $\beta_j$  of conical singularities, see the equalities~\eqref{phi123} and~\eqref{Phi} in Appendix~\ref{NP}. As a result, the definition~\eqref{FH} of $\mathcal H_{\pmb\beta}[\phi]$  immediately turns into an explicit formula for the function 
$(\beta_1,\beta_2,\beta_3)\mapsto \mathcal H_{\pmb\beta}[\phi]$. Therefore, in order to obtain an explicit formula for the determinant of Laplacian   $(\beta_1,\beta_2,\beta_3)\mapsto \det \Delta_{\pmb\beta}$
from the anomaly formula~\eqref{DetDelta1}, it remains to find an explicit expression for the Liouville action~\eqref{LA1}.  This is our next result. 

 \begin{theorem}[Explicit closed formula for Liouville action]\label{LAexplInt} 
For the Liouville action $\mathcal S_{\pmb\beta}[\phi]$ introduced in Theorem~\ref{THMdet} we have 
\begin{equation}\label{ExplLA}
\begin{aligned}
\frac 1 {4\pi} \mathcal S_{\pmb\beta}[\phi] = &-\frac{|\pmb\beta|+2}2( 2+ \log\pi)-\left(\frac{\beta_1^2+2\beta_1} 2-\frac{\beta_2^2+2\beta_2} 2  +\frac{\beta_3^2+2\beta_3} 2  \right)\log 2 
\\
&-\sum_{j=1}^3\Biggl(  \zeta_H'(-1,-\beta_j)+ \zeta_H'(-1,1+\beta_j)   
\\ 
&\qquad \qquad\qquad -\zeta_H'\left(-1,\beta_j-\frac{|\pmb \beta|}2\right) -\zeta_H'\left(-1,1+\frac{|\pmb \beta|}2-\beta_j\right) \Biggr)
\\
&\qquad \qquad\qquad+\zeta_H'\left(-1,-\frac{|\pmb\beta|}2\right) + \zeta_H'\left(-1,2+\frac{|\pmb\beta|}2\right)  -2\zeta_R'(-1).
\end{aligned}
\end{equation}
Here $\zeta_H(s,\nu)$ is the Hurwitz  and $\zeta_R(s)$ is the Riemann zeta function. The prime in  $\zeta'_H$ and $\zeta_R'$ stands for the derivative with respect to $s$. 
\end{theorem}

Now all the terms in the right hand side of the anomaly formula~\eqref{DetDelta1}  for the determinant $\det\Delta_{\pmb\beta}$ are explicitly expressed as functions of  the orders $\beta_j$ of conical singularities.  Let us stress, that in contrast to the determinant of Laplacian, the coefficients $\phi_j$ and the values of functionals $\mathcal S_{\pmb\beta}[\phi]$ and $\mathcal H_{\pmb\beta}[\phi]$ do depend on the normalization of marked points $p_j$; see Remark~\ref{GLA}.

As usual, the determinant of the Friedrichs Laplacian  $\Delta^S_{\pmb\beta}=\frac 1 {S} \Delta_{\pmb\beta}$, corresponding to the area $S$ metric $S\cdot e^{2\phi}|dz|^2$ (of Gaussian curvature $K=2\pi(|\pmb\beta|+2)/S$), is related to $\det\Delta_{\pmb\beta}$  by the standard  rescaling property 
\begin{equation}\label{rescaling1}
\log\det\Delta^S_{\pmb\beta}=\log\det\Delta_{\pmb\beta}-\zeta_{\pmb\beta}(0)\log S.
\end{equation}
Moreover, for the value of the spectral zeta function at zero we have 
 \begin{equation*}\label{Zeta_01}
\zeta_{\pmb\beta}(0)=\frac {2+|\pmb\beta|}{6}-\frac 1 {12}\sum_{j=1}^3\left(\beta_j+1-\frac 1 {\beta_j+1}\right)-1;
\end{equation*}
 a more detailed discussion can be found in Remark~\ref{Rescaling}. This together with the results above 
  allows one to write out an explicit closed formula for the determinant $\det\Delta^S_{\pmb\beta}$ in terms of the surface area $S$ and the orders $\beta_j$ of conical singularities. This is one of  the main results of this paper. 
  We formulate it as the following corollary of Theorem~\ref{THMdet} and Theorem~\ref{LAexplInt}:

\begin{corollary}[Explicit closed formula for the determinant]\label{COR}  Assume that $\beta_j\in(-1,0)$ and $\beta_j-\frac{|\pmb\beta|}2>0$ for  $j=1,2,3$. 
Let  $\Delta^S_{\pmb\beta}$ stand for the Friedrichs Laplacian  on the Riemann sphere $\overline{\Bbb C}$ equipped with (unique)  area $S$ constant  curvature   metric $S\cdot e^{2\phi}|dz|^2$ representing the divisor~\eqref{DIVISOR}.
Then for the zeta regularized spectral determinant $\det \Delta^S_{\pmb\beta}$ we have the  explicit closed formula 
\begin{equation}\label{DDAS}
\begin{aligned}
\log  {\det\Delta^S_{\pmb\beta}}=&- \frac {|\pmb\beta|+1} 6-\frac 1 {12\pi}\Bigl(\mathcal S_{\pmb\beta}[\phi]-\pi \log \mathcal H_{\pmb\beta}[\phi]\Bigr)
-\sum_{j=1}^{3}\mathcal C(\beta_j)
\\
&-\left(\frac {2+|\pmb\beta|}{6}-\frac 1 {12}\sum_{j=1}^3\left(\beta_j+1-\frac 1 {\beta_j+1}\right)-1\right)\log S
\\
& -\frac 4 3 \log 2 -4\zeta_R'(-1) -\log\pi.
 \end{aligned}
 \end{equation}
Here for the Liouville action $ \mathcal S_{\pmb\beta}[\phi]$ we have the explicit expression~\eqref{ExplLA},  the functional $\mathcal H_{\pmb\beta}[\phi]$ is explicitly defined via~\eqref{FH} with the coefficients $\phi_j$ found in~\eqref{phi123},~\eqref{Phi}. Finally, $\mathcal C(\beta_j)$ is explicitly given in~\eqref{Cb}. 
\end{corollary}

As we show in Subsection~\ref{FLAT_METRICS},  in the case $|\pmb\beta|=-2$ (i.e. when the metric $S\cdot e^{2\phi}|dz|$ is flat) the equality~\eqref{DDAS} significantly simplifies and generalizes our recent results in~\cite{KalvinJGA}.  
Moreover, in the limit cases of a spindle (i.e. when $\beta_j\to 0^-$ while $\beta_k=\beta_\ell$ for $k$ and $\ell$ different from $j$) and of a standard round sphere (i.e. when $\beta_j\to 0^-$ for $j=1,2,3$)  the formula~\eqref{DDAS}  returns the corresponding results known from~\cite{KalvinJFA,Klevtsov,SpreaficoZerbini}; for details we refer to Subsections~\ref{SubSpindle} and~\ref{StRoundSph}.

Our last theorem contains  results on stationary points  of the determinant.   
 \begin{theorem}[Stationary points of  determinant]\label{TSP} 
  Assume that $\beta_j\in(-1,0)$ and $\beta_j-\frac{|\pmb\beta|}2>0$ for  $j=1,2,3$. Then, on the constant curvature metrics  representing the divisors~\eqref{DIVISOR} of fixed degree $|\pmb\beta|=\beta_1+\beta_2+\beta_3$,
 the  point   $(\beta_1,\beta_2,\beta_3)$ with  
 \begin{equation}\label{SPI}
 \beta_1=\beta_2=\beta_3=\frac {|\pmb\beta|} 3
 \end{equation}
  is a stationary point of the function 
\begin{equation}\label{DasFA}
(\beta_1,\beta_2, \beta_3)\mapsto \log\det\Delta_{\pmb\beta}^S.
\end{equation}
Moreover, if the surface area $S$ is sufficiently small, then the stationary point is a minimum. More precisely: for each $|\pmb\beta|\in(-3,0)$ there exists a number $S_0=S_0(|\pmb\beta|)$ such that for any $S\in(0,S_0]$ 
 the stationary point~\eqref{SPI} is a minimum of the function~\eqref{DasFA}. 
 
\end{theorem}

Note that in~\cite{KalvinJGA} it is  demonstrated that the stationary point of the determinant on the flat isosceles triangle envelopes  (i.e. in the case $\beta_1=\beta_3=:\beta$ and $\beta_2=-2-2\beta$)   turns from a minimum to a maximum as the area $S$ increases. A similar effect also appears on the constant curvature metrics with two conical singularities on the $2$-sphere~\cite[Sec. 3.1]{KalvinJFA}. Based on the explicit formula for the determinant (Corollary~\ref{COR}) and its derivatives (see Proposition~\ref{SecDer}),  it is also  possible to clarify  (at least numerically) what happens with the minimum of  $\det\Delta_{\pmb\beta}^S$ when the area $S$ increases, however this goes out of the scope of this paper.

\section{Determinant of Laplacian}\label{DetLap}
In this section we prove Theorem~\ref{THMdet}. The proof relies on Proposition~\ref{PropDet} below.

\begin{proposition} \label{PropDet}
 For the determinant of the Friedrichs Laplacian  $\Delta_{\pmb\beta}$ on the Riemann sphere equipped with a  \underline{unit area}   constant curvature  metric $e^{2\phi}|dz|^2$ representing a divisor
$\pmb \beta=\sum_{j=1}^3 \beta_j\cdot p_j$ with three distinct points  $p_j\in \Bbb C$ and $\beta_j> -1$ we have 
\begin{equation}\label{DetDelta}
\begin{aligned}
\log
 {\det\Delta_{\pmb\beta}}
 = -\frac{|\pmb\beta|+2}{6}\int_{\Bbb C} \phi e^{2\phi}\,\frac {dz\wedge d\bar z} {-2i}   -\frac 1 3 \phi_\infty
  +\frac  1 6 \sum_{j=1}^{3}\frac {\beta_j}{\beta_j+1}\phi_j-\sum_{j=1}^{3}\mathcal C(\beta_j)
\\
 -\frac 4 3 \log 2 -4\zeta_R'(-1)+\frac 1 {6} -\log\pi.
\end{aligned}
\end{equation}
Here 
 $|\pmb\beta|=\beta_1+\beta_2+\beta_3$ is the degree of the divisor $\pmb\beta$, 
 $\phi_j$ and  $\phi_\infty$ are the coefficients in the asymptotics~\eqref{ASphi} of the metric potential $\phi$,  and $\mathcal C(\beta)$ is the function defined  in~\eqref{Cb}.
\end{proposition}

The proof of Proposition~\ref{PropDet} is preceded by Remark~\ref{RCb} and Remark~\ref{Rescaling}.   
 
\begin{remark} \label{RCb}
 The real analytic function $(-1,\infty)\ni\beta\mapsto \mathcal C(\beta)$ in~\eqref{Cb},~\eqref{DetDelta} describes the inputs into the determinant that come solely from the orders $\beta_j$ of conical singularities (the corresponding cone angles are $2\pi(\beta_j+1)$)~\cite{KalvinJFA}.
 
  The Barnes double zeta function $\zeta_B$  in~\eqref{Cb} is first defined by the double series
$$
\zeta_B(s;a,b,x)=\sum_{m,n=0}^\infty(am+bn+x)^{-s},\quad \Re s>2,a>0,b>0,x>0, 
$$
and then extended by analyticity to $s=0$, see e.g.~\cite{Matsumoto,Spreafico2}. In this paper we assume that $\beta<0$ and only  consider some limits as $\beta_j\to 0^-$ in Section~\ref{FSL}. In general, for a regular point  (i.e. if  there is no conical singularity at the point, or, equivalently, the cone angle is $2\pi$) we have $\beta=0$ and $\zeta'_B(0;\beta+1,1,1)=\zeta'_R(-1)$. Hence $\mathcal C(0)=0$ by~\eqref{Cb}.   
 
 For the rational values of $\beta$ the  function $\beta\mapsto \zeta'_B(0;\beta+1,1,1)$  in~\eqref{Cb} can be expressed in terms of the Riemann zeta and gamma functions. Namely, for any coprime natural numbers $p$ and $q$ we have
 \begin{equation}\label{BZ}
 \begin{aligned}
\zeta'_B(0;p/q, 1, 1 ) =  &\frac 1 {pq}\zeta_R'(-1)-\frac{1}{12pq}\log q  +\left(\frac 1 4   +  S(q,p)\right)\log\frac q  p 
\\
+\sum_{k=1}^{p-1}\left(\frac 1 2 -\frac k p \right)&\log \Gamma\left(  \left(\!\!\!\left(\frac{kq}{p}\right)\!\!\!\right)+\frac 1 2 \right)+\sum_{j=1}^{q-1}\left(\frac 1 2 -\frac j q \right)\log \Gamma\left(\left(\!\!\!\left(\frac{jp}{q}\right)\!\!\!\right)+\frac 1 2\right).
\end{aligned}
\end{equation}
Here $S(q,p)=\sum_{j=1}^{p}\left(\!\!\left(\frac{j}{p}\right)\!\!\right) \left(\!\!\left(\frac{jq}{p}\right)\!\!\right)$ is the Dedekind sum,  and   the symbol $(\!(\cdot)\!)$ is defined so that   $(\!(x)\!)=x-\lfloor x\rfloor-1/2$ for $x$ not an integer and $(\!(x)\!)=0$ for $x$ an integer (by $\lfloor x\rfloor$ we mean the floor of $x$: the largest integer not exceeding $x$).

Let us also note that in the case $p=1$ the equality~\eqref{BZ} simplifies to
$$
\zeta_B'\left(0; 1 /q,1,1\right) =  \frac 1 q \zeta'_R(-1)  -    \frac 1 {12q}\log q - \sum_{j=1}^{q-1}\frac j q    \log\Gamma\left(\frac j q\right)+\frac{q-1}4 \log 2\pi.
$$
For details we refer to~\cite[Appendix A]{KalvinJFA}.
 \end{remark}
 
\begin{remark}[Rescaling property]\label{Rescaling} Consider a unit area metric $m_{\pmb\beta}$ representing a divisor $\pmb \beta$. 
 Multiplying the metric $m_{\pmb\beta}$ by $S>0$, one obtains the metric  $m^S_{\pmb\beta}=S\cdot m_{\pmb\beta}$ of area $S$. The metric $m^S_{\pmb\beta}$ represents the same divisor $\pmb\beta$.  Let $\zeta_{\pmb\beta}(s)=\sum_j \lambda_j^{-s}$ stand for the spectral zeta function of the Friedrichs  Laplacian $\Delta_{\pmb\beta}$.  Then $\zeta^S_{\pmb\beta}(s)=\sum_j (\lambda_j/S)^{-s}$ is the spectral zeta function of the Friedrichs Laplacian $\Delta^S_{\pmb\beta}= \frac 1 S\cdot \Delta_{\pmb\beta}$ corresponding to the metric $m^S_{\pmb\beta}$. Differentiating $\zeta^S_{\pmb\beta}(s)$ with respect to $s$ we arrive at the standard rescaling property $\partial_s  \zeta^S_{\pmb\beta}(0)=\partial_s\zeta_{\pmb\beta}(0)+\zeta_{\pmb\beta}(0)\log S$.  Being rewritten in terms of the determinants it gives
\begin{equation}\label{rescaling}
\log\det\Delta^S_{\pmb\beta}=\log\det\Delta_{\pmb\beta}-\zeta_{\pmb\beta}(0)\log S.
\end{equation}
 For a sphere with three conical singularities  we have 
\begin{equation}\label{Zeta_0}
\zeta_{\pmb\beta}(0)=\frac {|\pmb\beta|+2}{6}-\frac 1 {12}\sum_{j=1}^3\left(\beta_j+1-\frac 1 {\beta_j+1}\right)-1;
\end{equation}
see~\cite[Corollary 1.2]{KalvinJFA}. As is well-known,  $a_0=\zeta_{\pmb\beta}(0)+1$ is actually the constant term in the short time asymptotic expansion of the heat trace $\Tr e^{-t\Delta_{\pmb\beta}}$, which allows one to deduce~\eqref{Zeta_0} from an expression for $a_0$, and vice versa. Furthermore, the heat  trace expansion identifies the poles of the spectral zeta function $s\mapsto \zeta_{\pmb\beta}(s)$. 
\end{remark}

\begin{proof}[Proof of Proposition~\ref{PropDet}]  Consider the Riemann sphere  $\overline{\Bbb C}$ equipped with the metric $e^{2\phi}|dz|^2$. We cut the sphere $\overline{\Bbb C}$ along the circle $|z|=R$ into the disk $|z|\leq R$ and the ``disk'' $\{|z|\geq R\}\cup\{\infty\}$.  By abuse of notation, we denote the  latter ``disk'' by $|z|\geq R$. 
Since we will let $R\to+ \infty$, we can assume that conical singularities $p_j\in \Bbb C$ of the metric $e^{2\phi}|dz|^2$ are always in the disk $|z|< R$. 

 The BFK decomposition formula~\cite[Theorem B$^*$]{BFK} gives
\begin{equation}\label{BFKf}
\det\Delta_{\pmb\beta}= \det (\Delta_{\pmb\beta}\!\!\restriction_{|z|\leq R}) \cdot \det (\Delta_{\pmb\beta}\!\!\restriction_{|z|\geq R})\cdot\frac {\det \left(\mathcal N_{\pmb\beta}\!\!\restriction_{|z|=R}\right)}{\ell(|z|= R,e^{2\phi}|dz|^2)}.
\end{equation}
Here $ \Delta_{\pmb\beta}\!\!\restriction_{|z|\leq R}$ and $\Delta_{\pmb\beta}\!\!\restriction_{|z|\geq R}$ are the Friedrichs Dirichlet Laplacians on the corresponding  disks equipped with the metric $e^{2\phi}|dz|^2$, $\ell(|z|= R,e^{2\phi}|dz|^2)$ is the metric length of the circle $|z|= R$, and  $\mathcal N_{\pmb\beta}\!\!\restriction_{|z|=R}$ is the Neumann jump operator on $|z|= R$ (a first-order classical pseudodifferential operator).  Note that for the constant curvature metrics with isolated conical singularities the proof of the BFK decomposition formula requires some minor modifications~\cite[Sec. 2.3]{KalvinJFA}.  

Similarly, for the decomposition of the standard unit sphere $(\overline {\Bbb C}, 4R^2(1+R^2|z|^2)^{-2}| dz |^2)$  along its equator $|z|=R$, the BFK formula reads
$$
\det\Delta= 4\pi \cdot \det (\Delta\!\!\restriction_{|z|\leq R}) \cdot \det (\Delta\!\!\restriction_{|z|\geq R})\cdot\frac {\det \left(\mathcal N\!\!\restriction_{|z|= R}\right)}{\ell(|z|= R, 4R^2(1+R^2|z|^2)^{-2}| dz |^2)}.
$$
Here $4\pi$ is the area of the unit sphere. (Note that  $(\overline {\Bbb C}, 4R^2(1+R^2|z|^2)^{-2}| dz |^2)$ is isometric to $(\overline {\Bbb C}, 4(1+|z|^2)^{-2}| dz |^2)$ via the change of variable $z\mapsto z/R$, and thus also to the standard unit sphere $x^2+y^2+z^2=1$ in $\Bbb R^3$ via the stereographic projection.) This together with the well-known explicit formulae for the the determinant $\det\Delta$ of the Laplacian on a unit sphere~\cite{OPS} and the determinant  $\det (\Delta\!\!\restriction_{|z|\leq R})= \det (\Delta\!\!\restriction_{|z|\geq R})$ of the Dirichlet Laplacian on a unit  hemisphere~\cite{Weisberger} allows one to conclude that
\begin{equation}\label{DNJ}
\frac {\det\left( \mathcal N_{\pmb\beta}\!\!\restriction_{|z|= R}\right)}  {\ell(|z|= R,e^{2\phi}|dz|^2)}=\frac {\det\left( \mathcal N\!\!\restriction_{|z|= R}\right)}{\ell(|z|= R,4R^2(1+R^2|z|^2)^{-2}| dz |^2)}=\frac 1 2.
\end{equation}
Here the first equality is valid due to the conformal invariance of the left hand side. The conformal invariance  
 can be most easily seen from the BFK formula together with Polyakov and Polyakov-Alvarez type formulas for the determinants of Laplacians, cf.~\cite[Proof of Lemma 2.2]{KalvinCCM}, see also~\cite{EW,GG,Went}.

Next we show that the determinant of the Dirichlet Laplacian on the smooth metric ``disk''    $\{|z|\geq R, e^{2\phi}|dz|^2\}$ obeys the asymptotics 
\begin{equation}\label{ASUD}\begin{aligned}
\log\det\left(\Delta_{\pmb\beta}\!\!\restriction_{|z|\geq R}\right)=       \frac 1 {3}\log R -&\frac 1 3 \phi_\infty
-2\zeta_R'(-1)
\\
&-\frac 5 {12} -\frac 1 6\log 2-\frac 1 2 \log \pi+o(1),\quad R\to \infty.
\end{aligned}
\end{equation}

Indeed, in the holomorphic coordinate $w=1/z$  the metric ``disk''  takes the form $\{|w|\leq1/R, e^{2\varphi}|dw|^2\}$ with a smooth (in the small disk $|w|\leq1/R$) metric potential $\varphi$ satisfying 
$$
\varphi(w)=\phi(1/w)-2\log |w| =\phi_\infty+o(1),\   |w|\partial_{ w}\varphi(w)=o(1) , \  |w|\partial_{\overline{ w}} \varphi(w) =o(1)  \text{  as  }  w\to 0;
$$ 
for the last two estimates we refer to~\cite[Lemma 3]{Troyanov}. 

Let us take the flat metric $|dw|^2$ as a reference metric in the disk $|w|\leq1/R$.
Then the  Polyakov-Alvarez formula~\cite{Alvarez,OPS} gives
\begin{equation}\label{UDP}
\begin{aligned}
\log \frac {\det\left(\Delta_{\pmb\beta}\!\!\restriction_{|z|\geq R}\right)}{ \det\left(\Delta_\flat\!\!\restriction_{|w|\leq 1/R}\right)}= -\frac{1}{6\pi}\left (\frac1 2 \int_{|w|\leq 1/R} |\nabla_\flat\varphi|^2\,\frac{dw\wedge d\overline{ w}}{-2i}+\oint_{|w|=1/R}k_\flat\varphi \,|d w|  \,\right )
\\
-\frac 1 {4\pi}\oint_{|w|=1/R}\partial_{n_\flat}\varphi\,|d w|.
\end{aligned}
\end{equation}
Here $\nabla_\flat$ is the gradient, $k_\flat=R$ is the geodesic curvature of the circle $|w|=1/R$,  and $n_\flat$ is the outward unit normal to the disk  $|w|\leq 1/R$ (all with respect  to the flat reference  metric $|dw|^2$). 

Notice that in the right hand side of the Polyakov-Alvarez formula~\eqref{UDP} only the integral involving the geodesic curvature $k_\flat$ gives a nonzero contribution of $-\frac 1 3 \phi_\infty +o(1)$ as $R\to \infty$, while all other integrals tend to zero. This together with the explicit formula 
\begin{equation}\label{WeisF}
\log  \det\left(\Delta_\flat\!\!\restriction_{|w|\leq 1/R}\right)=\frac 1  3 \log R +\frac 1 3 \log 2-\frac 1 2 \log 2\pi -\frac 5 {12} -2\zeta'_R(-1)
\end{equation}
for the selfadjoint Dirichlet Laplacian $\Delta_\flat\!\!\restriction_{|w|\leq 1/R}$ in the flat metric disk $\{|w|\leq1/R,|dw|^2\}$ establishes the asymptotics~\eqref{ASUD}; for the explicit formula~\eqref{WeisF} see~\cite[eqn.  (28)]{Weisberger}.  

We have studied the behaviour of the last two multiples in the right hand side of the BFK formula~\eqref{BFKf}  as $R\to+\infty$. Now we are in position to consider the first one.  

Take the flat metric $|dz|^2$ as a reference metric in the disk $|z|\leq R$. Since the metric $e^{2\phi}|dz|^2$ has three conical singularities in the disk $|z|<R$, no classical Polyakov-Alvarez formula like~\eqref{UDP} can be used. We rely on the (generalized) Polyakov-Alvarez type formula in~\cite[Theorem 1.1.2]{KalvinJFA} that is valid for a class of metrics with conical singularities, and, in particular, for the constant curvature metrics. The 
formula gives
\begin{equation}\label{PA_JFA}\begin{aligned}
&\log\frac{\det\left(\Delta_{\pmb\beta}\!\!\restriction_{|z|\leq R}\right)}{\det\left(\Delta_\flat\!\!\restriction_{|z|\leq R}\right)}=
-\frac{1}{12\pi}\left(\int_{|z|\leq R} K\phi e^{2\phi}\,\frac {dz\wedge d\bar z} {-2i}
 +\oint_{|z|=R} \phi\partial_{n_\flat}\phi\,|dz|\right)
  \\
  &-\frac 1{6\pi}\oint_{|z|=R}k_\flat\phi\, |dz|
-\frac 1 {4\pi}\oint_{|z|=R} \partial_{n_\flat}\phi\,|dz| +\frac  1 6 \sum_{j=1}^{3}\frac {\beta_j}{\beta_j+1}\phi_j-\sum_{j=1}^{3}\mathcal C(\beta_j).
\end{aligned}
\end{equation}
Here  $K=2\pi(|\pmb\beta|+2)$ is the (regularized) Gaussian curvature of the metric $e^{2\phi}|dz|^2$,  $ \partial_{ n_\flat}$   is the outer unit normal derivative with respect to the flat reference metric $|dz|^2$, and $k_\flat=1/R$  is the geodesic curvature of the circle $|z|=R$.  The function $\mathcal C$ is the same as in~\eqref{Cb} and Remark~\ref{RCb}.

 Let us stress that the equality~\eqref{PA_JFA}  is exactly the Polyakov-Alvarez type anomaly formula from~\cite{KalvinJFA}. Indeed, in~\cite[Theorem 1.1.2]{KalvinJFA}  we substitute $\psi\equiv 0$ for the potential of the Euclidean reference metric $|dz|^2$, $\phi$ for the metric potential $\varphi$ (as $\varphi=\phi-\psi=\phi$), $K$ for the regularized Gaussian curvature $K_\varphi$ of the metric $e^{2\varphi}|dz|^2$,  zero for the curvature $K_0$ of the reference metric $e^{2\psi}|dz|^2$, and $\psi_j(0)=0$ for the values of the potential $\psi\equiv 0$  at the points   $p_j$ in the support $\supp\pmb\beta$ of the divisor $\pmb\beta$.

By ~\cite[Lemma 3]{Troyanov}  the first order derivatives of the potential $\phi$ obey the estimates
\begin{equation}\label{est_phi'}
\begin{aligned}
 |z-p_j|\cdot\partial_{ z}(\phi(z)-\beta_j\log|z-p_j|)  =o(1)  \text{  as  }  z\to p_j, 
\\
|z-p_j|\cdot \partial_{\bar z}(\phi(z)-\beta_j\log|z-p_j|)   =o(1)  \text{  as  }  z\to p_j,
 \\
|z|\cdot  \partial_{ z}(\phi(z)+2\log|z|)   =o(1)  \text{  as  }  z\to \infty,
\\
 |z|\cdot\partial_{\bar z}(\phi(z)+2\log|z|)   =o(1)  \text{  as  }  z\to\infty.
\end{aligned}
\end{equation}

Thanks to the asymptotics~\eqref{ASphi} and the estimates~\eqref{est_phi'} for the metric potential  $\phi$ in vicinities of the  points  $p_j$ and at infinity, we conclude that  as $R\to+\infty$ the contour integrals in~\eqref{PA_JFA} meet the following estimates:
\begin{equation}\label{PA_JFAest}
\begin{aligned}
-\frac 1 {12\pi}\oint_{|z|=R}\phi\partial_{ n_\flat}\phi\, |dz|=
\frac{1} 6\Bigl( - 4\log R+2\phi_\infty\Bigr)+o(1),
\\
-\frac 1 {4\pi}\oint_{|z|=R}\partial_{ n_\flat}\phi\, |dz|= 1+ o(1),
\\
-\frac 1 {6\pi}\oint_{|z|=R}k_\flat\phi \,|dz|=
 \frac {2} 3 \log R -\frac 1 3 \phi_\infty +o(1).
\end{aligned}
\end{equation}

As in~\eqref{WeisF}, for the determinant of the selfadjoint Dirichlet Laplacian on the disk $|z|\leq R$ endowed with the flat reference metric $|dz|^2$ we have  
\begin{equation}\label{WeisF2}
\log  \det\left(\Delta_\flat\!\!\restriction_{|z|\leq R}\right)=-\frac 1  3 \log R +\frac 1 3 \log 2-\frac 1 2 \log 2\pi -\frac 5 {12} -2\zeta'_R(-1).
\end{equation}

Summing up,     from~\eqref{PA_JFA} together with~\eqref{PA_JFAest} and~\eqref{WeisF2}  we obtain  the asymptotics
\begin{equation}\label{ASLD}
\begin{aligned}
\log{\det\left(\Delta_{\pmb\beta}\!\!\restriction_{|z|\leq R}\right)}=-\frac{|\pmb\beta|+2}{6}\int_{\Bbb C} \phi e^{2\phi}\,\frac {dz\wedge d\bar z} {-2i} 
+1  +\frac  1 6 \sum_{j=1}^{3}\frac {\beta_j}{\beta_j+1}\phi_j -\sum_{j=1}^{3}\mathcal C(\beta_j)\\
-\frac 1  3 \log R +\frac 1 3 \log 2-\frac 1 2 \log 2\pi -\frac 5 {12} -2\zeta'_R(-1)+o(1),\ R\to+\infty.
\end{aligned}
\end{equation}
This completes our study of the behaviour of the  multiples in the right hand side of the BFK formula~\eqref{BFKf} as $R\to+\infty$.

Now we are in position to pass to the limit in the BFK formula~\eqref{BFKf}  as $R\to+\infty$. Taking into account~\eqref{DNJ} together with  asymptotics~\eqref{ASUD} and~\eqref{ASLD} we rewrite the BFK formula~\eqref{BFKf} in the form
$$
\begin{aligned}
\log \det\Delta_{\pmb\beta}=
-&\frac{|\pmb\beta|+2}{6}\int_{\Bbb C} \phi e^{2\phi}\,\frac {dz\wedge d\bar z} {-2i} 
+1  +\frac  1 6 \sum_{j=1}^{3}\frac {\beta_j}{\beta_j+1}\phi_j-\sum_{j=1}^{3}\mathcal C(\beta_j)
\\
&-\frac 1  3 \log R +\frac 1 3 \log 2-\frac 1 2 \log 2\pi -\frac 5 {12} -2\zeta'_R(-1)
\\
      & +\frac 1 {3}\log R -\frac 1 3 \phi_\infty
-2\zeta_R'(-1)-\frac 5 {12} -\frac 1 6\log 2-\frac 1 2 \log \pi
\\
& -\log 2+o(1), \quad R\to+\infty.
\end{aligned}
$$
Combining the like terms and passing to the limit we arrive at the  anomaly formula~\eqref{DetDelta}. 
\end{proof}

For the constant curvature unit area metrics $e^{2\phi}|dz|^2$ with three conical singularities one can obtain not only explicit formulas for the coefficients $\phi_j$ in the asymptotics~\eqref{LEq1}, but also for the metric potential  $\phi$ itself, see Remark~\ref{EFphi} in Appendix~\ref{NP}.  
Now, when we have  the anomaly formula~\eqref{DetDelta} for the determinant of Laplacian at hands, 
one can naively try to  substitute the explicit expression~\eqref{EEphi} for $\phi$  into the integral in~\eqref{DetDelta}, with a hope to obtain an explicit formula for the determinant of Laplacian after the integration.   Unfortunately this  
 plan does not seem to be realistic. Except for the curvature zero case, when  $|\pmb\beta|=-2$ and the integral term in~\eqref{DetDelta} disappears, see~Remark~\ref{FLATLA} in Section~\ref{SecGE} below. In order to study the general case, we introduce the Liouville action $\mathcal S_{\pmb\beta}[\phi]$. Then, closely following  original ideas of A. Zamolodchikov and Al. Zamolodchikov~\cite{Z-Z},  we obtain the explicit formula~\eqref{ExplLA} for $\mathcal S_{\pmb\beta}[\phi]$.

\begin{proof}[Proof of Theorem~\ref{THMdet}] The assertion of theorem is an immediate consequence of 
Proposition~\ref{PropDet} together with the definition~\eqref{LA1} for the Liouville action $\mathcal S_{\pmb\beta}[\phi]$  and the definition~\eqref{FH} for the functional $\mathcal H_{\pmb\beta}[\phi]$.

Indeed, taking into account the definitions of  $\mathcal S_{\pmb\beta}[\phi]$ and $\mathcal H_{\pmb\beta}[\phi]$, one can rewrite the anomaly formula~\eqref{DetDelta} for $\log\det\Delta_{\pmb\beta}$  in the equivalent form~\eqref{DetDelta1}. This 
completes the proof of Theorem~\ref{THMdet}. 
\end{proof}

Let us also mention that in Proposition~\ref{PropDet} and Theorem~\ref{THMdet} we restrict ourselves to the case of three conical singularities only because the main purpose of this paper is to deduce an explicit closed formula for the determinant. Starting from four conical singularities on a $2$-sphere an explicit construction of the uniformazation map is an open long standing  problem, while in the case of three conical singularities the construction is essentially well-known, cf. Appendix~\ref{NP}. Thus we do not known how to obtain general explicit closed formulae for the coefficients $\phi_j$ and the Liouville action in the case of $n>3$ conical singularities. However, this makes no obstruction towards obtaining an anomaly formula for the Determinant of Laplacian similar to~\eqref{DetDelta} (or~\eqref{DetDelta1}) for any number of conical singularities, the proof remains essentially the same.

We end this section with a remark that an analog of the anomaly formula~\eqref{DetDelta}  for the determinant  can also be obtained even for non-constant curvature  singular metrics. For instance, the proof of Proposition~\ref{PropDet} does not require any significant changes if the (regularized) Gaussian curvature $K$ of a unit area metric $e^{2\phi}|dz|^2$  is any smooth function that is constant only in vicinities of the conical singularities $p_j\in \Bbb C$ of  $e^{2\phi}|dz|^2$. In this case  the metric potential $\phi$ satisfies the Liouville equation~\eqref{LEq1} with $2\pi(|\pmb\beta|+2)$ replaced by $K=K(z)$ in the right hand side,  and $\phi$ has the asymptotics~\eqref{ASphi} with some coefficients $\phi_j$ and $\phi_\infty$.  
As a consequence, the anomaly formula~\eqref{PA_JFA} is still valid~\cite{KalvinJFA}, but the Gaussian curvature $K=K(z)$ in it cannot be replaced by $2\pi(|\pmb \beta|+2)$ anymore.  The resulting analog of the anomaly formula~\eqref{DetDelta}  reads
$$
\begin{aligned}
\log
 {\det\Delta_{\phi}}
 = -\frac 1{12\pi}\int_{\Bbb C}  K \phi e^{2\phi}\,\frac {dz\wedge d\bar z} {-2i}   -\frac 1 3 \phi_\infty
  +\frac  1 6 \sum_{j=1}^{3}\frac {\beta_j}{\beta_j+1}\phi_j-\sum_{j=1}^{3}\mathcal C(\beta_j)
\\
 -\frac 4 3 \log 2 -4\zeta_R'(-1)+\frac 1 {6} -\log\pi.
\end{aligned}
$$
This anomaly formula can also be written in terms of  $\mathcal S_{\pmb\beta}[\phi]$,      $\mathcal H_{\pmb\beta}[\phi]$, and   $\mathcal C(\beta_j)$ as in Theorem~\ref{THMdet}:  
The definitions of $\mathcal H_{\pmb\beta}[\phi]$ and $\mathcal C(\beta_j)$ remain the same, but for the Liouville action one has to take a more general definition. For instance,   the Liouville action can be defined via the equality
\be\label{LAgeneral}
\mathcal S_{\pmb\beta}[\phi]:=  \int_{\Bbb C} K \phi e^{2\phi}  \frac {dz\wedge d\bar z }{-2i} - 2\pi(|\pmb\beta|+2) +2\pi\sum \beta_j\phi_j+4\pi\phi_\infty.
\ee
We shall not dwell upon this and will restrict ourselves to the unit area metrics of constant curvature $K=2\pi(|\pmb\beta|+2)$ as before.

Let us stress that the results of this section do not depend on the normalization of  the distinct marked points $p_j\in \Bbb C$.  When formulating the main results of this paper in Section~\ref{MainResults} we made a particular choice of the  normalization only to simplify the exposition.

\section{Liouville action}\label{LiAc}

In this section we find the Liouville action $\mathcal S_{\pmb\beta}[\phi]$   in a closed explicit form. In particular, we prove Theorem~\ref{LAexplInt} and Corollary~\ref{COR}.  Thus in Subsection~\ref{SecGE} we show that the Liouville action $\mathcal S_{\pmb\beta}[\phi]$ satisfies a system of governing differential equations. Then in Subsection~\ref{LA_S2} we integrate the system and find the constant of integration. This constitutes the proof of Theorem~\ref{LAexplInt}.  As a consequence of Theorem~\ref{THMdet} (proved in Section~\ref{DetLap}) and Theorem~\ref{LAexplInt} (that we prove in Subsection~\ref{LA_S2} below) we  obtain Corollary~\ref{COR}. 

To simplify the exposition we normalize the marked points  $p_j\in\Bbb C$ so that  $p_j=0,\pm1$ and consider only the metrics of unit area (this is exactly what we need in order to prove Theorem~\ref{LAexplInt} and Corollary~\ref{COR}). The (minor) modifications required to obtain similar results for the Liouville action evaluated on an arbitrary  constant curvature metric with three conical singularities
 are briefly discussed in Remark~\ref{GLA}.

\subsection{Governing equations}\label{SecGE}

\begin{lemma}[After A. Zamolodchikov \& Al. Zamolodchikov]\label{ZZ}

The Liouville action   $\mathcal S_{\pmb\beta}[\phi]$ introduced in Theorem~\ref{THMdet} 
satisfies the system of governing  differential equations 
\begin{equation}\label{GovEqn}
\partial_{\beta_j}\mathcal S_{\pmb\beta}[\phi]=4\pi \phi_j-2\pi,\quad j=1,2,3.
\end{equation}
Here $\phi_1$, $\phi_2$, and $\phi_3$ are the coefficients in the asymptotics~\eqref{ASphi} of $\phi$.
\end{lemma}

\begin{proof}[Proof of Lemma~\ref{ZZ}]
Denote 
$$
\Bbb C_R:=\{z\in\Bbb C: 1/R\leq |z|\leq R, |z-1|\geq 1/R, |z+1|\geq 1/R  \}.
$$
Since $\phi$ satisfies the Liouville equation~\eqref{LEq1}, for the  integral in the definition~\eqref{LA1} of the Liouville action $\mathcal S_{\pmb\beta}[\phi]$ we have
$$
\begin{aligned}
 2\pi(|\pmb\beta|+2)\int_{\Bbb C}  \phi e^{2\phi}  \frac {dz\wedge d\bar z }{-2i}&=\lim_{R\to+\infty} \int_{\Bbb C_R} (-4\partial_z\partial_{\bar z}\phi) \phi  \frac {dz\wedge d\bar z }{-2i}
\\
&=\lim_{R\to +\infty}\left(\int_{\Bbb C_R}4|\phi_z|^2\frac{dz\wedge d\bar z}{-2i}-i\int_{\partial\Bbb C_R} \bigl(\phi\phi_{\bar z}\,d\bar z -\phi\phi_z\, dz\bigr)\right).
\end{aligned}
$$
As a consequence of the estimates~\eqref{est_phi'} for the first order derivatives of $\phi$, we obtain
$$
\begin{aligned}
\lim_{R\to +\infty}&\left(\int_{\Bbb C_R}4|\phi_z|^2\frac{dz\wedge d\bar z}{-2i}-i\int_{\partial\Bbb C_R} \bigl(\phi\phi_{\bar z}\,d\bar z -\phi\phi_z\, dz\bigr)\right)
\\
&=\lim_{R\to +\infty}\left(\int_{\Bbb C_R}4|\phi_z|^2\frac{dz\wedge d\bar z}{-2i}+i\oint_{|z|=R}\phi\left(\frac{d\bar z }{\bar z} -\frac {dz}{z}\right) \right.
\\
&\qquad\qquad \qquad\left.  - \frac i 2  \sum_{j=1}^{3}\beta_j \oint_{|z-p_j|=1/R} \phi\left(\frac{d\bar z}{\bar z-\bar p_j} -\frac {dz}{z-p_j}\right) \right).
\end{aligned}
$$
The contour integrals above satisfy the estimates
\begin{equation}\label{aux}
-\frac i 2  \oint_{|z-p_j|=1/R} \phi\left(\frac{d\bar z}{\bar z-\bar p_j} -\frac {dz}{z-p_j}\right)+2\pi\beta_j\log R =2 \pi\phi_j +o(1),\quad R\to+\infty,
\end{equation}
$$
i\oint_{|z|=R}\phi\left(\frac{d\bar z }{\bar z} -\frac {dz}{z}\right) +8\pi\log R
= 4 \pi\phi_\infty+o(1),\quad R\to+\infty.
$$
Indeed, these estimates easily follow from the behaviour~\eqref{ASphi} of $\phi$ near the marked points~$p_j$  and at infinity.

Summing up, we can  rewrite the Liouville action~\eqref{LA1} in the following equivalent form:
\begin{equation}\label{LFTZ}
\begin{aligned}
\mathcal S_{\pmb\beta}[\phi]
&=\lim_{R\to +\infty}\Biggl(\int_{\Bbb C_R}(4|\phi_z|^2-Ke^{2\phi})\frac{dz\wedge d\bar z}{-2i}+2i\oint_{|z|=R}\phi\left(\frac{d\bar z }{\bar z} -\frac {dz}{z}\right)
\\
&-  i   \sum_{j=1}^{3}\beta_j \oint_{|z-p_j|=1/R} \phi\left(\frac{d\bar z}{\bar z-\bar p_j} -\frac {dz}{z-p_j}\right) + 2\pi\log R\sum\beta_j^2+ 8\pi\log R  \Biggr).
\end{aligned}
\end{equation}
Here $K=2\pi(|\pmb\beta|+2)$ is the (regularized) Gaussian curvature of the unit area singular metric $ e^{2\phi}|dz|^2$.

Differentiating~\eqref{LFTZ} with respect to $\beta_\ell$, and taking into account the Liouville equation~\eqref{LEq1} together with the asymptotics~\eqref{aux}, we obtain
\be\label{auxSGE}
\begin{aligned}
\partial_{\beta_\ell}\mathcal S_{\pmb\beta}&[\phi]
=\lim_{R\to +\infty}\Biggl(\int_{\Bbb C_R}(4\partial_{\beta_\ell}|\phi_z|^2-2Ke^{2\phi}\partial_{\beta_\ell}\phi)\frac{dz\wedge d\bar z}{-2i} -  ( \partial_{\beta_\ell}K) \int_{\Bbb C_R}e^{2\phi}\frac{dz\wedge d\bar z}{-2i} 
\\
&+2i\oint_{|z|=R}\partial_{\beta_\ell}\phi\left(\frac{d\bar z }{\bar z} -\frac {dz}{z}\right) 
 -  i   \sum_{j=1}^{3}\beta_j \oint_{|z-z_j|=1/R}\partial_{\beta_\ell} \phi\left(\frac{d\bar z}{\bar z-\bar z_j} -\frac {dz}{z-z_j}       \right) 
 \\
&\qquad \qquad -  i   \oint_{|z-z_\ell|=1/R} \phi\left(\frac{d\bar z}{\bar z-\bar z_\ell} -\frac {dz}{z-z_\ell}\right) +4\pi  \beta_\ell \log R \Biggr)
\\
&=2 \lim_{R\to +\infty}\left(\int_{\Bbb C_R}(-4\partial_{z}\partial_{\bar z} \phi-Ke^{2\phi})\partial_{\beta_\ell}\phi\frac{dz\wedge d\bar z}{-2i} \right)+4\pi\phi_\ell -   \partial_{\beta_\ell}K
\\
&=4\pi\phi_\ell -  \partial_{\beta_\ell}K.
\end{aligned}
\ee
Since $\partial_{\beta_\ell}K=2\pi$, this completes the proof of lemma.
\end{proof}

\begin{remark}\label{FLATLA} Consider the particularly simple curvature zero case. Then $|\pmb\beta|=-2$ and  the flat  metric $e^{2\phi}|dz|^2$ on the Riemann sphere $\overline{\Bbb C}$ can be written in the form  
$$
e^{2\phi}|dz|^2= C_{\pmb\beta}^2|z+1|^{2\beta_1}|z|^{2\beta_2} |z-1|^{2\beta_3}|dz|^2,
$$
see e.g.~\cite{OPS1,TroyanovSSC} and Appendix~\ref{NP}.  The scaling factor $C^2_{\pmb\beta}$  ensures that $e^{2\phi}|dz|^2 $ is a unit area metric, i.e.
$$
C_{\pmb\beta}=\left(\int_{\Bbb C} |z+1|^{2\beta_1}|z|^{2\beta_2} |z-1|^{2\beta_3}\frac {dz\wedge d\bar z}{-2i} \right)^{-1/2};
$$
recall that we normalize the marked points so that $p_1=-1$, $p_2=0$, and $p_3=1$.

Clearly, the metric potential $\phi$ satisfies the Liouville equation~\eqref{LEq1}  and has the asymptotics~\eqref{ASphi} with 
\begin{equation}\label{Potflat}
\phi_1=\beta_3 \log 2+\log C_{\pmb\beta},
\quad \phi_2=\log C_{\pmb\beta},
\quad
\phi_3=\beta_1\log 2 +\log C_{\pmb\beta},
\quad
\phi_\infty=\log C_{\pmb\beta}.
\end{equation}

Since $|\pmb\beta|=-2$, the first term in the right hand side of the definition~\eqref{LA1} for the  Liouville action $\mathcal S_{\pmb\beta}[\phi]$  disappears. As a result we immediately obtain
\begin{equation}\label{LiouvFlat}
\mathcal S_{\pmb\beta}[\phi]\bigr|_{|\pmb\beta|=-2}=\left.\left(2\pi\sum_{j=1}^3 \beta_j\phi_j+4\pi\phi_\infty\right) \right|_{\beta_2=-2-\beta_1-\beta_3}=4\pi\beta_1\beta_3 \log 2.
\end{equation}
In accordance with Lemma~\ref{ZZ} we should have
$$
\partial_{\beta_1}\left(\mathcal S_{\pmb\beta}[\phi]\bigr|_{\beta_2=-2-\beta_1-\beta_3} \right)=(\partial_{\beta_1} \mathcal S_{\pmb\beta}[\phi]-\partial_{\beta_2}\mathcal S_{\pmb\beta}[\phi] )\bigr|_{\beta_2=-2-\beta_1-\beta_3}=4\pi(\phi_1-\phi_2)\bigr|_{\beta_2=-2-\beta_1-\beta_3},
$$
$$
\partial_{\beta_3}\left(\mathcal S_{\pmb\beta}[\phi]\bigr|_{\beta_2=-2-\beta_1-\beta_3} \right)=(\partial_{\beta_3} \mathcal S_{\pmb\beta}[\phi]-\partial_{\beta_2}\mathcal S_{\pmb\beta}[\phi])\bigr|_{\beta_2=-2-\beta_1-\beta_3} =4\pi(\phi_3-\phi_2)\bigr|_{\beta_2=-2-\beta_1-\beta_3}.
$$
This is in agreement with~\eqref{Potflat} and~\eqref{LiouvFlat}.
\end{remark}

\begin{remark} \label{ZZTZ} 
In order to see that  $\mathcal S_{\pmb\beta}[\phi]$ is $4\pi$ times  the Liouville action  introduced by A. Zamolodchikov and Al. Zamolodchikov   in~\cite[Eq. (2.34), where $\varphi=2\phi$, $K=-4\pi\mu b^2$,  and $\eta_j=-\beta_j/2$]{Z-Z}, one need only  note that the contour integrals in~\eqref{LFTZ} can equivalently be represented in the form
$$
i\oint_{|z|=R}\phi\left(\frac{d\bar z }{\bar z} -\frac {dz}{z}\right)= \frac{1}{\pi R}   \oint_{|z|=R}\phi\, ds,
$$
\begin{equation*}
-\frac i 2  \oint_{|z-p_j|=1/R} \phi\left(\frac{d\bar z}{\bar z-\bar p_j} -\frac {dz}{z-p_j}\right)=-\frac R {2\pi}  \oint_{|z-p_j|=1/R} \phi\,{ds}.
\end{equation*}
 The equality~\eqref{LFTZ} also shows that our definition of the Liouville action is in agreement with those in~\cite{CMS,HJ,T-Z}.
 
 Let us also note  that the system of governing equations~\eqref{GovEqn} is slightly different from the one in~\cite[Eq. (4.8)]{Z-Z}. This is because  in~\cite{Z-Z}  the authors consider the metrics $e^{2\phi}|dz|^2$ with thee conical singularities and fixed Gaussian curvature $K=-4\pi\mu b^2<0$, where $b$ is the dimensionless Liouville coupling constant and  $\mu$ is the so-called cosmological constant.  For these metrics the Liouville action is defined via the general formula~\eqref{LAgeneral}  (clearly,~\eqref{LAgeneral} specializes to~\eqref{LA1} if  $K=2\pi(|\pmb\beta|+2)$), or, equivalently, via~\eqref{LFTZ} with $K=-4\pi\mu b^2$.  As a result,  $\partial_{\beta_\ell} K=0$, and the equalities~\eqref{auxSGE} imply 
 $$
 \partial_{\beta_j}   \mathcal S_{\pmb\beta}[\phi]=4\pi \phi_j,\quad j=1,2,3,
 $$ 
which is equivalent to~\cite[Eq. (4.8)]{Z-Z}. For a fixed Gaussian curvature $K$ the flat case with the Liouville action~\eqref{LiouvFlat} corresponds to the limit  as the area $S=2\pi(|\pmb\beta|+2)/K$ of the corresponding metrics $e^{2\phi}|dz|^2$ goes to zero, the cases of positive and negative $K$ have to be studied separately.
\end{remark}

\begin{remark}The Liouville equation~\eqref{LEq1} (with $p_1=-1$, $p_0=0$, and $p_1$=1)
is the Euler-Lagrange equation for the Liouville action functional $\psi\mapsto\mathcal S_{\pmb\beta}[\psi]$ defined via~\eqref{LA1}. 
Indeed, it is not hard to verify that for any $\eta\in C^\infty(\overline{\Bbb C}, \Bbb R)$ one has
$$
\lim_{t\to 0}\frac {\mathcal S_{\pmb\beta}[\phi+t\eta]-\mathcal S_{\pmb\beta}[\phi]}{t}=2\int_{\Bbb C} ( -4\partial_z\partial_{\bar z}\phi- 2\pi(|\pmb\beta|+2)e^{2\phi})\eta\frac{dz\wedge d\bar z}{-2i}.
$$
Thus the functional  $\psi\mapsto\mathcal S_{\pmb\beta}[\psi]$ has a non-degenerate critical point given by the potential $\phi$ satisfying the Liouville equation~\eqref{LEq1} and having the asymptotics~\eqref{ASphi} ; cf. \cite[p.589]{Z-Z},~\cite[Remark 4]{T-Z}.
\end{remark}

\subsection{Liouville action: Explicit expression}\label{LA_S2}
In this subsection we obtain the closed explicit formula~\eqref{ExplLA} for the Liouville action $\mathcal S_{\pmb\beta}[\phi]$ introduced in Theorem~\ref{THMdet}.  Namely, we integrate the system of governing differential equations~\eqref{GovEqn} and find the constant of integration by using Remark~\ref{FLATLA} above.   This constitutes the proof of Theorem~\ref{LAexplInt}, where we utilize the explicit expressions for the coefficients $\phi_j$ obtained in Appendix~\ref{NP}, see  Proposition~\ref{METRIC}.

\begin{proof}[Proof of Theorem~\ref{LAexplInt}] Let us integrate the first governing equation~\eqref{GovEqn} with $\phi_1$ replaced by its explicit expression in terms of $\beta_j$ found in~\eqref{phi123},~\eqref{Phi}.

Integrating the first term  in the right hand side of~\eqref{Phi}  we get
$$
\begin{aligned}
& \frac 1 2 \int \log \frac{\Gamma\left(2+{|\pmb \beta|}/2\right)}{ 4\pi\Gamma\left(-{|\pmb \beta|} /2\right)} \,d\beta_1 =-\frac{\beta_1}{2}\log(4\pi)+\psi^{(-2)}(-|\pmb\beta|/2)+\psi^{(-2)}(2+|\pmb\beta|/2)+C
\\
 & =-\frac{\beta_1}{2}\log(4\pi)+ \zeta_H'\left(-1,-\frac{|\pmb\beta|}2\right)  + \zeta_H'\left(-1,2+\frac{|\pmb\beta|}2\right) -\frac {|\pmb\beta|^2}4 -|\pmb\beta|+C.
\end{aligned}
$$
Here and elsewhere for the (generalized) polygamma function $\psi^{(-2)}$ we use the identity
$$
\begin{aligned}
\psi^{(-2)}(x)={\zeta_H'(-1,x)+(\gamma+\psi(2))\zeta_H(-1,x)}
= \zeta_H'(-1,x) -\frac{B_2(x)} 2,
\end{aligned}
$$
 where $B_2(x)=\frac 1  6 - x + x^2$ is the second Bernoulli polynomial, see e.g.~\cite[Eq. (2.3)]{pgamma}.

For the second term  in the right hand side of~\eqref{Phi} we obtain
$$
\begin{aligned}
\int  \log \frac {\Gamma(-\beta_1)}{\Gamma(1+\beta_1)}\,d\beta_1&=-\psi^{(-2)}(-\beta_1)-\psi^{(-2)}(1+\beta_1)+C
\\
&= \beta_1^2 +\beta_1  - \zeta_H'(-1,-\beta_1)-    \zeta_H'(-1,1+\beta_1) +C.
\end{aligned}
$$

Towards the integration of the third term  in the right hand side of~\eqref{Phi}, we first notice that
$$
\begin{aligned}
&\int \log \frac{ \Gamma(\beta_1-|\pmb \beta|/2) }{   \Gamma(1+|\pmb \beta|/2-\beta_1)  }\,d\beta_1
=2\psi^{(-2)} (\beta_1-|\pmb \beta|/2)+2\psi^{(-2)} (1+|\pmb \beta|/2-\beta_1)+C
\\
&=2\zeta_H'(-1,\beta_1-|\pmb \beta|/2) +2\zeta_H'(-1,1+|\pmb \beta|/2-\beta_1) -2(|\pmb \beta|/2-\beta_1)^2-2(|\pmb \beta|/2-\beta_1)+C.
\end{aligned}
$$
Similarly we obtain 
$$
\begin{aligned}
&\int \log \frac{ \Gamma(1+|\pmb \beta|/2-\beta) }{   \Gamma(\beta-|\pmb \beta|/2)  }\,d\beta_1
=2\psi^{(-2)}(\beta-|\pmb \beta|/2)+2\psi^{(-2)}(1+|\pmb \beta|/2-\beta)+C
\\
&=2\zeta_H'(-1,\beta-|\pmb \beta|/2) +2\zeta_H'(-1,1+|\pmb \beta|/2-\beta) -2(|\pmb \beta|/2-\beta)^2-2(|\pmb \beta|/2-\beta)+C,
\end{aligned}
$$
where either $\beta=\beta_2$ or $\beta=\beta_3$. 
In total,  for the third term  in the right hand side of~\eqref{Phi} we have $$
\begin{aligned}
&\frac 1 2\int  \log \frac
{\Gamma(\beta_1-|\pmb \beta|/2)
 \Gamma(1+|\pmb \beta|/2-\beta_2)  \Gamma(1+|\pmb \beta|/2-\beta_3) }{   \Gamma(1+|\pmb \beta|/2-\beta_1)    \Gamma(\beta_2-|\pmb \beta|/2) \Gamma(\beta_3-|\pmb \beta|/2) }\,d\beta_1
\\
 &=\sum_{j=1}^3 \left(\zeta_H'\left(-1,\beta_j-\frac{|\pmb \beta|}2\right) +\zeta_H'\left(-1,1+\frac{|\pmb \beta|}2-\beta_j \right)- \left(\frac{|\pmb \beta|}2-\beta_j\right)^2 \right)-\frac {|\pmb \beta|}2+C.
 \end{aligned}
$$

Thus integration of the first governing equation in~\eqref{GovEqn} gives
$$
\begin{aligned}
\frac 1 {4\pi} \mathcal S_{\pmb\beta}[\phi]=-\frac {\beta_1^2}{2}\log 2
-\frac{\beta_1}{2}\log(4\pi)+ \zeta_H'\left(-1,-\frac{|\pmb\beta|}2\right)  + \zeta_H'\left(-1,2+\frac{|\pmb\beta|}2\right) -\frac {|\pmb\beta|^2}4 -|\pmb\beta|
\\
+\beta_1^2 +\beta_1  - \zeta_H'(-1,-\beta_1)-    \zeta_H'(-1,1+\beta_1) 
\\+\sum_{j=1}^3 \left(\zeta_H'\left(-1,\beta_j-\frac{|\pmb \beta|}2\right) +\zeta_H'\left(-1,1+\frac{|\pmb \beta|}2-\beta_j \right)- \left(\frac{|\pmb \beta|}2-\beta_j\right)^2 \right)-\frac {|\pmb \beta|}2+C,
\end{aligned}
$$
where the constant of integration $C=C(\beta_2,\beta_3)$ does not depend on $\beta_1$.  

The other two governing  equations~\eqref{GovEqn} can be integrated in exactly the same way.  The required explicit expressions for the coefficients $\phi_2$ and $\phi_3$ are also given in~\eqref{phi123},~\eqref{Phi}. We omit the details. 
 
Summing up, we obtain
\begin{equation}\label{ZZS}
\begin{aligned}
\frac 1 {4\pi} \mathcal S_{\pmb\beta}[\phi]=  -\frac{|\pmb\beta|}2\bigl(1+\log(4\pi)\bigr)+\left(-\frac{\beta_1^2} 2 +\frac{\beta_2^2} 2+2\beta_2  -\frac{\beta_3^2} 2\right)\log 2
\\
+\zeta_H'\left(-1,-\frac{|\pmb\beta|}2\right)  + \zeta_H'\left(-1,2+\frac{|\pmb\beta|}2\right) -\frac {|\pmb\beta|^2}4 -\frac {3|\pmb\beta|}2
\\
+\sum_{j=1}^3 (\beta_j^2 +\beta_j)  - \sum_{j=1}^3 \bigl(\zeta_H'(-1,-\beta_j)+ \zeta_H'(-1,1+\beta_j)\bigr) -\sum_{j=1}^3 \left(\frac{|\pmb \beta|}2-\beta_j\right)^2
\\
+\sum_{j=1}^3 \left(\zeta_H'\left(-1,\beta_j-\frac{|\pmb \beta|}2\right) +\zeta_H'\left(-1,1+\frac{|\pmb \beta|}2-\beta_j\right) \right)+C,
\end{aligned}
\end{equation}
where the constant of integration $C$ does not depend on the orders $\beta_j$ of conical singularities. 
The equality~\eqref{ZZS} simplifies to 
$$
\begin{aligned}
\frac 1 {4\pi} \mathcal S_{\pmb\beta}[\phi]=  -|\pmb\beta| -\frac{|\pmb\beta|}2   \log\pi-\left(\frac{\beta_1^2+2\beta_1} 2-\frac{\beta_2^2+2\beta_2} 2  +\frac{\beta_3^2+2\beta_3} 2  \right)\log 2
\\
-\sum_{j=1}^3\Biggl(  \zeta_H'(-1,-\beta_j)+ \zeta_H'(-1,1+\beta_j)    
-\zeta_H'\left(-1, \beta_j-\frac{|\pmb \beta|}2\right) -\zeta_H'\left(-1,1+\frac{|\pmb \beta|}2-\beta_j\right) \Biggr)
\\
+\zeta_H'\left(-1,-\frac{|\pmb\beta|}2\right)  + \zeta_H'\left(-1,2+\frac{|\pmb\beta|}2\right) +C.
\end{aligned}
$$
In the flat case   the latter equality  takes the form
$$
\frac 1 {4\pi} \mathcal S_{\pmb\beta}[\phi]\Bigr|_{|\pmb\beta|=-2}=2+\log\pi+\beta_1\beta_3\log 2 
+2\zeta_R'(-1)+C.
$$
This together with~\eqref{LiouvFlat} allows one to find the constant of integration:
$$
C=-2-\log \pi-2\zeta_R'(-1).
$$
This completes the proof of Theorem~\ref{LAexplInt}.
\end{proof}

\begin{proof}[Proof of Corollary~\ref{COR}]
The assertion is an immediate consequence of Theorem~\ref{THMdet} and Theorem~\ref{LAexplInt}; see  Proposition~\ref{METRIC}  in Appendix~\ref{NP} and the standard rescaling property in Remark~\ref{Rescaling}.  
\end{proof}

\begin{remark}\label{GLA}  Consider a  (unique) metric $S\cdot e^{2\phi}|dz|^2$ of area $S$ representing a divisor $\pmb\beta=\sum_{j=1}^3\beta_j\cdot p_j$, where $\beta_j\in(-1,0)$ and $p_j\in\Bbb C$ are any three distinct points.  For this metric  the Liouville action can be defined via the formula~\eqref{LAgeneral} (or, equivalently, via~\eqref{LFTZ}) with $\phi$ replaced by  $\phi+\frac 1 2 \log S$. It is then easy to check that 
$$
\mathcal S_{\pmb\beta}\left[\phi+\frac 1 2 \log S\right]=\mathcal S_{\pmb\beta}[\phi]+2\pi(|\pmb\beta|+2)\log S.
$$
Here $\mathcal S_{\pmb\beta}[\phi]$ is the Liouville action defined for the unit  area metric $e^{2\phi}|dz|^2$ by the equality~\eqref{LA1}, where the coefficients $\phi_j$ and $\phi_\infty$ are the same as in the asymptotics~\eqref{ASphi}.  

Note that in Lemma~\ref{ZZ} we assumed that $p_j=0,\pm1$ only to simplify the exposition:  the result and its proof remain valid for any normalization of the marked points $p_j\in\Bbb C$. Since the metric potentials for any two normalizations of the marked points $p_j$ are related by means of a M\"obius transformation, for the coefficient~$\phi_1=\phi_1(\pmb\beta)$ in~\eqref{ASphi} one has 
\begin{equation}\label{mpotpj}
\begin{aligned}
\phi_1 & = (\beta_1+1)\log\left|\frac {p_3-p_2}{(p_3-p_1)(p_2-p_1)}\right|+\log 2 +  \Phi(\beta_1,\beta_2,\beta_3);
\end{aligned}
\end{equation}
 cf. Proposition~\ref{METRIC}. Clearly, the corresponding expressions for $\phi_2=\phi_2(\pmb\beta)$ and $\phi_3=\phi_3(\pmb\beta)$ can be obtained by  permutations of $\beta_j$ and $p_j$. Thus only minor changes are required to  the proof of  Theorem~\ref{LAexplInt}  in order to include into consideration the case of any three distinct points $p_j\in\Bbb C$ and any area $S$.
In particular, for a (fixed) Gaussian curvature $K$ and the area $S=2\pi(|\pmb\beta|+2)/K$ we obtain 

\begin{equation}\label{ZZaction}
\begin{aligned}
&\frac 1 {4\pi}\mathcal S_{\pmb\beta}\left[\phi+\frac 1 2 \log \frac{2\pi(|\pmb\beta|+2)}K\right] = \mathcal S^{(cl)}(K;\beta_1,\beta_2,\beta_3)+\frac{\delta_1+\delta_2-\delta_3}2\log|p_1-p_2|\\
&
+\frac{\delta_2+\delta_3-\delta_1}2\log|p_2-p_3|+\frac{\delta_3+\delta_1-\delta_2}2\log|p_1-p_3|,
\end{aligned}
\end{equation}
where $\delta_j=-\beta_j(\beta_j+2)$.  The classical  Liouville action $\mathcal S^{(cl)}(K;\beta_1,\beta_2,\beta_3)$ is  explicitly defined via the equality
\begin{equation}\label{ZZclaction}
\begin{aligned}
\mathcal S^{(cl)}(K;\beta_1,\beta_2,\beta_3)&= \frac{|\pmb\beta|+2}2\left(\log\frac {2(|\pmb\beta|+2)}{K}-2\right)
\\
&-\sum_{j=1}^3\Biggl(  \zeta_H'(-1,-\beta_j)+ \zeta_H'(-1,1+\beta_j)   
\\ 
&\qquad \qquad -\zeta_H'\left(-1,\beta_j-\frac{|\pmb \beta|}2\right) -\zeta_H'\left(-1,1+\frac{|\pmb \beta|}2-\beta_j\right) \Biggr)
\\
&\qquad \qquad+\zeta_H'\left(-1,-\frac{|\pmb\beta|}2\right) + \zeta_H'\left(-1,2+\frac{|\pmb\beta|}2\right)  -2\zeta_R'(-1).
\end{aligned}
\end{equation}
The classical Liouville action does not depend on the marked points $p_j$,  the dependence of the Liouville action on the marked points is completely described by the last three terms in the right hand side of~\eqref{ZZaction}.

 This is intimately connected with the celebrated DOZZ formula of H. Dorn, H.-J. Otto~\cite{DO} and   A. Zamolodchikov, Al. Zamolodchikov~\cite{Z-Z}. The DOZZ formula is a heuristically deduced explicit expression  for the three-point structure constant $C(\alpha_1,\alpha_2,\alpha_3)$ of the Liouville conformal field theory. In the classical limit the leading (exponential) asymptotics  of  the three-point structure constant  is governed by the classical Liouville action. Namely, in accordance with~\cite[Eq. (3.20)]{Z-Z},  the structure constant has the leading asymptotics
$$
C(-\beta_1/2b,-\beta_2/2b,-\beta_3/2b)\sim\exp\left(-\frac 1 { b^2} \mathcal S^{(cl)} (-4\pi\mu b^2;\beta_1,\beta_2,\beta_3) \right)
$$
as the Liouville coupling constant $b$ goes to zero. Here $\mu>0$ is the so-called  cosmological constant, and hence the Gaussian curvature $K=-4\pi\mu b^2$   is negative.  

Note that the right hand side of~\eqref{ZZaction} (resp. of ~\eqref{ZZclaction}) is an explicit expression for the Liouville action found in~\cite[Eq.~(4.12)]{Z-Z}   (resp. for the classical Liouville action found in~\cite[Eq. (3.21)]{Z-Z}) in a different form, where $\eta_j=-\beta_j/2$ and $|x_{ij}|=|p_i-p_j|$. The DOZZ formula itself received a mathematical interpretation and proof only recently~\cite{KRV-DOZZ}.

With the help of the equalities~\eqref{mpotpj} and~\eqref{ZZaction} one can also check that the right hand side of the anomaly formula for the determinant of Laplacian~\eqref{DetDelta1} does not depend on the particular choice of normalization for the marked points~$p_j$ indeed.
\end{remark}

\section{Determinant  for flat and  limit spherical metrics}\label{FSL}

\subsection{Flat metrics}\label{FLAT_METRICS} In the case $|\pmb\beta|=-2$ the Gauss-Bonnet theorem implies $K=2\pi(2+|\pmb\beta|)=0$. Thus we deal with a flat  (Gaussian curvature $K=0$) singular surface ---  a Euclidean surface with conical singularities in the sense of~\cite{TroyanovSSC}, see also~\cite{H1,H2,H3,HK,KalvinJGA,OPS1}.  The surface can be visualized as a triangle envelope: a Euclidean triangle with internal angles $\pi(\beta_j+1)$ glued  along the edges to its reflection in a side, see Appendix~\ref{NP} for more detail. 
Let us also recall that the space of uniform metrics~\cite{OPS1}  on a sphere with three distinct open disks removed can be identified with a subset of the space  
of all flat metrics with three conical singularities.

In~\cite{KalvinJGA} we studied the determinant of Friedrichs Laplacians on the Euclidean isosceles triangle envelopes. In particular, we derived  a closed explicit formula for the zeta-regularized spectral determinant in terms of the angles and the total area of the surface~\cite[Prop. 3.1 and Eqs. (7.1), (7.2)]{KalvinJGA}.  The isosceles triangle envelopes correspond to the restriction $\beta_1=\beta_3=\beta$ and $\beta_2=-2-2\beta$, where $\beta\in(-1,-1/2)$. Proposition~\ref{ETEnv} below allows for   arbitrary orders of conical singularities.

\begin{proposition}\label{ETEnv} Let $\beta_1+\beta_2+\beta_3=-2$ with $\beta_j\in(-1,0)$. Consider the  flat metric
$$
S\cdot C_{\pmb \beta}^2 |z+1|^{2\beta_1} |z|^{2\beta_2}|z-1|^{2\beta_3}|dz|^2
$$
of area $S$ on the Riemann sphere $\overline{\Bbb C}$, where $C^2_{\pmb \beta}$ is the scaling factor 
$$
C^2_{\pmb\beta}=2^{2\beta_2+2}\frac {\Gamma(-\beta_1)\Gamma(-\beta_2)\Gamma(-\beta_3)}{\pi\Gamma(\beta_1+1)\Gamma(\beta_2+1)\Gamma(\beta_3+1)}.
$$
Then for the zeta-regularized spectral determinant of the corresponding Friedrichs  Laplacian $\Delta^S_{\pmb\beta}\Bigr|_{|\pmb\beta|=-2}$ we have the explicit closed formula
\begin{equation}\label{DetEuclEnv}
\begin{aligned}
\log  {\det\Delta^S_{\pmb\beta}}\Bigr|_{|\pmb\beta|=-2} &= \frac{1}{6}\left( \frac{\beta_1\beta_3}{\beta_1+1} + \frac{\beta_1\beta_3}{\beta_3+1}\right)\log 2 
 \\
 -&  \sum_{j=1}^{3}\left(2\zeta'_B(0;\beta_j+1,1,1) -\frac {\beta_j^2}{6(\beta_j+1)}\log 2 +\frac 1 2 \log(\beta_j+1)\right)
 \\
 -&\log C^2_{\pmb\beta}-\left(\zeta_{\pmb\beta}(0)\Bigr|_{|\pmb\beta|=-2}\right)\log (   C^2_{\pmb\beta} S)
-\frac 4 3 \log 2 +2\zeta_R'(-1) -\log\pi.
 \end{aligned}
\end{equation}
 Here $\zeta'_B$ and $\zeta_R'$ are the derivatives with respect to $s$ of the Barnes double zeta function $\zeta_B(s;a,b,x)$ and the Riemann zeta function $\zeta_R(s)$ respectively, and 
\begin{equation}\label{Zeta_0_flat}
\zeta_{\pmb\beta}(0)\Bigr|_{|\pmb\beta|=-2}=-\frac {13}{12}+\frac 1 {12}\sum_{j=1}^3\frac 1{\beta_j+1}. 
\end{equation}
\end{proposition}
\begin{proof} The equality~\eqref{DetEuclEnv} is a special case of the one in Corollary~\ref{COR}. 

Indeed, for the flat case the  coefficients $\phi_j$ in the asymptotics of the metric potential $\phi$ and the Liouville action $\mathcal S_{\pmb\beta}[\phi]$  were already found in Remark~\ref{FLATLA}. In addition, for the functional $ \mathcal H_{\pmb\beta}[\phi]$ defined in~\eqref{FH} we obtain
\begin{equation}\label{10:34}
\begin{aligned}
\frac 1 {12}\log \mathcal H_{\pmb\beta}[\phi]\Bigr|_{|\pmb\beta|=-2}
=&\frac 1 6\left(\beta_1+1-\frac 1 {\beta_1+1}\right)\beta_3\log 2
\\
&+\frac 1 6\left(\beta_3+1-\frac 1 {\beta_3+1}\right)\beta_1\log 2-(\zeta_{\pmb\beta}(0)+1)\log C^2_{\pmb\beta}.
\end{aligned}
\end{equation}
Here the expression $-(\zeta_{\pmb\beta}(0)+1)$ in the right hand side   represents the sum
$$
\frac 1 {12}\sum_j\left(\beta_j+1-\frac 1 {\beta_j+1}\right)=\frac {2+|\pmb\beta|}{6}-\zeta_{\pmb\beta}(0)-1=-(\zeta_{\pmb\beta}(0)+1),\quad |\pmb\beta|=-2.
$$
As a consequence of the explicit expression~\eqref{LiouvFlat} for the Liouville action $\mathcal S_{\pmb\beta}[\phi]|_{|\pmb\beta|=-2}$ and the formula~\eqref{10:34}  for the functional  $ \mathcal H_{\pmb\beta}[\phi]|_{|\pmb\beta|=-2}$ we get
$$
-\frac 1 {12\pi}\Bigl(\mathcal S_{\pmb\beta}[\phi]-\pi \log \mathcal H_{\pmb\beta}[\phi]\Bigr)\Bigr|_{|\pmb\beta|=-2}= \frac{1}{6}\left( \frac{\beta_1\beta_3}{\beta_1+1} + \frac{\beta_1\beta_3}{\beta_3+1}\right)\log 2 -(\zeta_{\pmb\beta}(0)+1)\log C^2_{\pmb\beta}.
$$
 After some algebra 
 this together with Corollary~\ref{COR} implies the equality~\eqref{DetEuclEnv}.
Moreover, in the case $|\pmb\beta|=-2$ the formula~\eqref{Zeta_0} for $\zeta_{\pmb\beta}(0)$ reduces to the one in~\eqref{Zeta_0_flat}.
\end{proof}

The closed explicit formula~\eqref{DetEuclEnv} for the determinants of Laplacians on the flat triangle envelopes generalizes our previous results in~\cite[Prop. 3.1 and formulae (7.1), (7.2)]
{KalvinJGA}. It is also interesting note that the celebrated partially heuristic Aurell-Salomonson formula~\cite[(50)]{AS2}  returns a result equivalent to the one in Proposition~\ref{ETEnv}. For  details and a rigorous mathematical proof of the Aurell-Salomonson formula we refer to~\cite[Sec.~3.2]{KalvinJFA}.  For the most recent progress towards obtaining a rigorous mathematical proof of a similar Aurel-Salomonson formula for polygons we refer to~\cite{AKR} and references therein.

\subsection{Spherical metrics with two antipodal singularities} \label{SubSpindle}
  A (unique) constant curvature unit area metric $e^{2\varphi}|dz|^2$ on $\overline{\Bbb C}$ representing the divisor  $\pmb \beta= \beta_1\cdot 0+\beta_2\cdot 1+\beta_3\cdot \infty$, where $\beta_j\in(-1,0)$, can be written in the form
\be\label{fin}
e^{2\varphi}|dz|^2= 4  (1+2\pi(|\pmb \beta|+2)|w( z)|^2)^{-2}  |w'(  z)|^2\,|d z|^2,
\ee
where $w( z)$ is the  Schwarz triangle function; see Appendix~\ref{NP}. Here we consider the limit case  $\beta_1=\beta_3=\beta$ as $\beta_2\to 0^-$.  

In the limit we have  $|\pmb\beta|=2\beta>-2$, $w( z)=c_{\pmb\beta} z^{\beta+1}$ with a scaling coefficient  $c_{\pmb\beta}$, and the metric takes the form
$$
e^{2\varphi}|dz|^2= \frac{4    c_{\pmb\beta}^2  (\beta+1)^2 | z|^{2\beta} \,|d z|^2}{(1+4\pi(\beta+1)c^2_{\pmb\beta}| z|^{2\beta+2})^2}.
$$
The change of variable $z\mapsto  z/c_{\pmb\beta}$ shows that  this is the metric of a spindle of Gaussian curvature $K=4\pi(\beta+1)$: a spherical metric with two antipodal conical singularities of order $\beta$ representing the divisor $\beta\cdot 0+\beta\cdot\infty$, see~\cite{TroyanovSp}. 

We apply the M\"obius transformation  $  z\mapsto f(z)=\frac{1+z}{1-z}$ to the metric $e^{2\varphi}|dz|^2$ in order to pass to the metric  $e^{2\phi}|dz|^2$ representing the divisor $\beta_1\cdot(-1)+\beta_2\cdot 0+\beta_3\cdot 1$.  The discussion above shows that  in the limit  $\beta_2\to 0^-$  the resulting metric $e^{2\phi}|dz|^2$ with $\phi=\varphi\circ f+\log|f'|$ and $\beta_1=\beta_3=\beta$ turns into the unit area Gaussian curvature $K=4\pi(\beta+1)$ metric of a spindle. The corresponding  divisor is 
 $
\pmb\beta=\beta\cdot(-1)+\beta\cdot 1$. 

For the coefficients in the asymptotics of the  metric potential near the conical singularities we have 
$$
\begin{aligned}
\phi_1\bigr|_{\beta_1=\beta_2=\beta, \beta_2\to 0^-} =\phi_3\bigr|_{\beta_1=\beta_2=\beta, \beta_2\to 0^-} = -\beta\log 2 +\Phi(\beta,\beta_2,\beta)\bigr|_{\beta_2\to 0^-}
\\
=-\beta\log 2+\frac 1 2 \log\frac {\beta+1} {4\pi},
\end{aligned}
$$
see Proposition~\ref{METRIC}. Hence for the functional $\mathcal H_{\pmb\beta}[\phi]$ in~\eqref{FH} we obtain
\begin{equation}\label{Hspindle}
\log \mathcal H_{\pmb\beta}[\phi]\Bigr|_{\beta_1=\beta_3=\beta,\beta_2\to 0^-}=4\left(\beta+1-\frac{1}{\beta+1}\right)\left( -\beta\log 2+ \frac 1 2 \log\frac {\beta+1} {4\pi}  \right).
\end{equation}

In order to pass to the limit  in the explicit expression for the Liouville action in Theorem~\ref{LAexplInt} as $\beta_2\to 0$,   we use the identities
$$
\zeta_H'\left(-1,\beta-\frac {|\pmb\beta|} 2\right)\bigr|_{\beta_1=\beta_3=\beta}=\zeta_H'\left(-1,1-\frac {\beta_2} 2\right)+\frac {\beta_2} 2\log\left(-\frac {\beta_2} 2 \right),
$$
$$
\zeta_H'(-1,-\beta_2)=\zeta_H'(-1,1-\beta_2)+\beta_2\log(-\beta_2),
$$
$$
\zeta_H'(-1,2+\beta)=\zeta_H'(-1,1+\beta)+(1+\beta)\log(1+\beta)
$$
that easily follow from the definition $\zeta_H(s,\nu)=\sum_{n=0}^\infty (n+\nu)^{-s}$ of the Hurwitz zeta function. 
As a result, the explicit formula~\eqref{ExplLA} for the  Liouville action takes the form 
\begin{equation*}
\begin{aligned}
\frac 1 {4\pi} \mathcal S_{\pmb\beta}[\phi]\bigr|_{\beta_1=\beta_2=\beta, \beta_2\to 0} = \lim_{\beta_2\to 0}\Bigl(-(\beta+1)( 2+ \log\pi)-\left(\beta^2+2\beta \right)\log 2 
\\ +\beta_2\log\left(-\frac {\beta_2} 2 \right)-\beta_2\log(-\beta_2) +(1+\beta)\log(1+\beta)  \Bigr)
\\
= -(\beta+1)( 2+ \log\pi)  -\left({\beta^2+2\beta}  \right)\log 2+(1+\beta)\log (1+\beta).
\end{aligned}
\end{equation*}
This together with~\eqref{Hspindle} implies
\begin{equation}\label{SpindleUA}
\frac 1 {4\pi}\Bigl(\mathcal S_{\pmb\beta}[\phi]-\pi \log \mathcal H_{\pmb\beta}[\phi]\Bigr)\Bigr|_{\beta_1=\beta_3=\beta,\beta_2\to 0^-}
=\frac 1 2 \left(\beta+1+\frac 1 {\beta+1}\right)\log\frac{\beta+1}\pi-2(\beta+1).
\end{equation}

In Theorem~\ref{THMdet} the  anomaly formula for the determinant of Laplacian~\eqref{DetDelta1} together with the definition~\eqref{Cb}  for the function $\mathcal C(\beta)$  give
\begin{equation*}
\begin{aligned}
\log  {\det\Delta}_{\pmb\beta}\Bigr|_{\beta_1=\beta_3=\beta,\beta_2\to 0^-}
=- \frac {\beta+1} 6-\frac 1 {12\pi}\Bigl(\mathcal S_{\pmb\beta}[\phi]-\pi \log \mathcal H_{\pmb\beta}[\phi]\Bigr)\Bigr|_{\beta_1=\beta_3=\beta,\beta_2\to 0}
\\
-4\zeta'_B(0;\beta+1,1,1)+\frac {\beta^2}{3(\beta+1)}\log 2 - \log(\beta+1)
 -\frac 4 3 \log 2  -\log\pi.
 \end{aligned}
\end{equation*}
Taking into accont~\eqref{SpindleUA} we finally arrive at the equality
\begin{equation}\label{DetSpindle}
\begin{aligned}
\log  {\det\Delta}_{\pmb\beta}\Bigr|_{\beta_1=\beta_3=\beta,\beta_2\to 0^-}=\frac {\beta+1} 2 -\frac 1 6 \left( \beta+1+\frac{1}{\beta+1}\right)\log\frac{\beta+1}{{2\pi}} 
\\
-4\zeta_B'(0;\beta+1,1,1) -\log \bigl(4\pi(\beta+1)\bigr).
\end{aligned}
\end{equation}
This is exactly the formula for the determinant of Laplacian on the spindle, cf.~\cite[Prop. 3.1]{KalvinJFA}    (where for the unit area spindle with two antipodal singularities one should take $K_\varphi=4\pi(\beta+1)$ and $\mu=0$), see also~\cite{KalvinCCM,SpreaficoZerbini} and~\cite[Appendix B]{Klevtsov}.

The limit cases $\beta_1=\beta_2=\beta$ as $\beta_3\to 0^-$ and $\beta_2=\beta_3=\beta$ as $\beta_1\to 0^-$ are similar and lead to exactly the same results. We omit the details.

Let us also note that the explicit expressions for the determinant of  Friedrichs  Dirichlet Laplacians on the constant curvature cones~\cite[Sec. 3.3]{KalvinJFA} and, in particular, the one for the flat cones~\cite{Spreafico},  can be independently obtained as a consequence of the equality~\eqref{DetSpindle}; for details we refer to~\cite{KalvinCCM}.

\subsection{Standard round sphere}\label{StRoundSph}

The limit case $\beta_1=\beta_2=\beta_3\to 0^-$ corresponds to the standard round sphere.  
Indeed, in this case the Schwarz triangle function $w(z)$ in~\eqref{fin}  takes the form  $w( z)=c_{\pmb\beta} z$ and thus 
$$
e^{2\varphi}|dz|^2 = \frac{4\,|c_{\pmb\beta}d z|^2}{(1+4\pi  | c_{\pmb\beta} z|^2)^2}.
$$
The change of variable $z\mapsto z/c_{\pmb\beta}$ brings this metric  into the form of standard curvature $4\pi$ unit area metric of a sphere. Therefore the Riemann sphere with the metric $4\pi e^{2\varphi}|dz|^2$ is isometric to the standard round  sphere $x_1^2+x_2^2+x_3^2=1$ in $\Bbb R^3$.  

 Below we show that in the limit $\beta_j\to 0^-$ our  formula for the determinant in Corollary~\ref{COR}   returns the well-known value of the determinant of Laplacian on the standard  round  sphere (of Gaussian curvature one and area $4\pi$). 

Indeed, as a consequence of Theorem~\ref{LAexplInt} we conclude that  
\begin{equation}\label{ActionSphere}
\frac 1 {4\pi} \mathcal S_{\pmb\beta}[\phi]\to -2-\log\pi,\quad \beta_j\to 0^-, \ j=1,2,3.
\end{equation}
For $\beta_j=0$ we also have $\mathcal H_{\pmb\beta}[\phi]=1$ and $\mathcal C(\beta_j)=0$, see~\eqref{FH} and Remark~\ref{RCb}. As a result the explicit formula for $\log  {\det\Delta^S_{\pmb\beta}}$ from Corollary~\ref{COR} takes the form
\begin{equation*}
\begin{aligned}
\log  {\det\Delta^{4\pi}_{\pmb\beta}}\bigr|_{\beta_j\to 0^-}=- \frac {1} 6+\frac {2+\log\pi } {3}
-\left(\frac {2}{6}-1\right)\log(4\pi) -\frac 4 3 \log 2 -4\zeta_R'(-1) -\log\pi
\\
= \frac {1} 2   -4\zeta_R'(-1). 
 \end{aligned}
 \end{equation*}
This is exactly the  LogDet of the Laplacian on the standard round sphere, see e.g.~\cite{OPS}.

\section{Stationary points of determinant}\label{SecSPoints}

In this section we study stationary points of the determinant, deduce explicit formulas for its  second derivatives, and, in particular, prove Theorem~\ref{TSP}.

\begin{proposition}\label{StatPoints}[Stationary points] Consider the determinant of Laplacian on the constant curvature  metrics
representing the divisor
$$
\pmb\beta=\beta_1\cdot(-1)+\beta_2\cdot 0+\beta_3\cdot 1,\quad \beta_j\in(-1,0),
$$
of fixed degree $|\pmb\beta|=\beta_1+\beta_2+\beta_3$. Then the point   $(\beta_1,\beta_2,\beta_3)$ with  $\beta_1=\beta_2=\beta_3=\frac {|\pmb\beta|} 3$ is a stationary point of the function 
$
(\beta_1,\beta_2, \beta_3)\mapsto \log\det\Delta_{\pmb\beta}^S$.

Note that the Gauss-Bonnet theorem implies $|\pmb\beta|=\frac {SK}{2\pi}-2$, where $S$ is the surface area and $K$ is the Gaussian curvature.
\end{proposition}

\begin{proof} Let us first consider only the metrics of unit area, i.e.  we assume that $S=1$ and  ${ K}= {2\pi}(|\pmb\beta|+2)$ with a fixed $|\pmb\beta|$. Without loss of generality  we can  set $\beta_2=|\pmb\beta|-\beta_1-\beta_3$  and  consider the determinant as a function of two variables:  $\beta_1$ and $\beta_3$.  Thus we need to show that 
 the equations 
\be\label{STPb}
\partial_{\beta_\ell} \left(\log\det\Delta_{\pmb\beta}\Bigr|_{\beta_2=|\pmb\beta|-\beta_1-\beta_3}\right)=0,\quad \ell=1,3,
\ee
are satisfied if $\beta_1=\beta_3=\frac {|\pmb\beta|} 3$. Recall that  in the case  $S=1$ we denote the Laplacian $\Delta^S_{\pmb\beta}$ by $\Delta_{\pmb\beta}$.

As a consequence of the formula~\eqref{DetDelta} for $\log\det\Delta_{\pmb\beta}$ and  the governing equations~\eqref{GovEqn} for the Liouville action, the equations~\eqref{STPb} 
can equivalently be written in the form
\begin{equation}\label{CPoint}
\begin{aligned}
-\frac 1 3 (\phi_\ell-\phi_2)\Bigr|_{\beta_2=|\pmb\beta|-\beta_1-\beta_3}   
+&\frac 1 {12}     \partial_{\beta_\ell}\left(\log\mathcal H_{\pmb\beta}[\phi]\Bigr|_{\beta_2=|\pmb\beta|-\beta_1-\beta_3}\right)
\\
 &-\mathcal C'(\beta_\ell)+\mathcal C'\left(|\pmb\beta|-\beta_1-\beta_3\right)=0,\quad \ell=1,3.
  \end{aligned}
\end{equation}
Here $\phi_j$ are the functions found in~\eqref{phi123},~\eqref{Phi}, the functional $\mathcal H_{\pmb\beta}[\phi]$ is defined in~\eqref{FH}, and    $\mathcal C(\beta)$ is the same as in~\eqref{Cb}.

It is not hard to verify  that  for $\beta_1=\beta_3=\frac {|\pmb\beta|}3$  and  $\ell=1,3$ we have
$$
-\mathcal C'(\beta_\ell)+\mathcal C'\left( |\pmb\beta|-\beta_1-\beta_3\right)=-\mathcal C'\left(\frac {|\pmb\beta|}3\right)+\mathcal C'\left(\frac {|\pmb\beta|}3\right)=0,
$$
$$
\phi_\ell-\phi_2= -2\frac{|\pmb\beta|+3}{3}\log 2,\quad \partial_{\beta_\ell}\left(\log\mathcal H_{\pmb\beta}[\phi]\Bigr|_{\beta_2=|\pmb\beta|-\beta_1-\beta_3}\right)=  -8 \frac{|\pmb\beta|+3}{3}\log 2.
$$
Therefore the equations~\eqref{CPoint}, or, equivalently, the equations~\eqref{STPb} are satisfied. This completes the proof for the metrics  $m_{\pmb\beta}$ of unit area.

Multiplying $m_{\pmb\beta}$  by $S$ one obtains the metric  $S\cdot m_{\pmb\beta}$ of area $S$ and Gaussian curvature $K=2\pi(|\pmb\beta|+2)/S$.  The corresponding determinant of Laplacian is related to $\det\Delta_{\pmb\beta}$ by the standard rescaling property, see Remark~\ref{Rescaling}.  In the case $\beta_j=\frac{|\pmb\beta|}3 =\frac {SK}{6\pi}-\frac 2 3$  for the value of the spectral zeta function at zero we have   
 $$
 \partial_{\beta_\ell}\left(\zeta_{\pmb\beta}(0)\Bigr|_{\beta_2=|\pmb\beta|-\beta_1-\beta_3}\right)=0,\quad\ell=1,3;
 $$
cf.~\eqref{Zeta_0}. This together with the rescaling property~\eqref{rescaling} immediately implies that the point $(\beta_1,\beta_2,\beta_3)$ with $\beta_1=\beta_2=\beta_3=\frac{|\pmb\beta|}3=\frac {SK}{6\pi}-\frac 2 3$ is a stationary point of the function $(\beta_1,\beta_2, \beta_3)\mapsto \log\det\Delta_{\pmb\beta}^S$ on the constant curvature metrics with a fixed value of $|\pmb\beta|$. 
\end{proof}

\begin{proposition}[Second derivatives]\label{SecDer}  As in Proposition~\ref{StatPoints}  above, consider the determinant of Laplacian on the constant curvature  metrics 
representing the divisor
$$
\pmb\beta=\beta_1\cdot(-1)+\beta_2\cdot 0+\beta_3\cdot 1,\quad \beta_j\in(-1,0),
$$
of fixed degree $|\pmb\beta|=\beta_1+\beta_1+
\beta_3$. For the second derivatives of  $\log\det\Delta^S_{\pmb\beta}$ evaluated at the stationary point $\beta_j=\frac {|\pmb\beta|} 3=\frac {SK}{6\pi}-\frac 2 3$ we have 
$$
\begin{aligned}
&\partial^2_{\beta_1} \left(\log\det\Delta^S_{\pmb\beta}\Bigr|_{\beta_2=|\pmb\beta|-\beta_1-\beta_3}\right)\Bigr|_{\beta_j=\frac {|\pmb\beta|}{3}} =\partial^2_{\beta_3} \left(\log\det\Delta^S_{\pmb\beta}\Bigr|_{\beta_2=|\pmb\beta|-\beta_1-\beta_3}\right)\Bigr|_{\beta_j=\frac {|\pmb\beta|}{3}} 
\\
&=
-9\left({|\pmb\beta|+3}\right)^{-3}( 2\log 2+\log S +2\Phi \Bigr|_{\beta_j=\frac {|\pmb\beta|}{3}}   )    +  6 \left( {|\pmb\beta|+3}\right)^{-2}    (\partial_{\beta_1}\Phi      -\partial_{\beta_2}\Phi)\Bigr|_{\beta_j=\frac {|\pmb\beta|}{3}} 
\\
&\qquad+\frac 1 3 \left( \frac {|\pmb\beta|+3}{3}-\frac {3}{|\pmb\beta|+3}\right)  ( \partial^2_{\beta_1}\Phi+\partial^2_{\beta_2}\Phi+\partial^2_{\beta_3}\Phi )\Bigr|_{\beta_j=\frac {|\pmb\beta|}{3}} -2\mathcal C''\left(\frac {|\pmb\beta|}{3}\right)
\end{aligned}
$$
and 
$$
\begin{aligned}
&\partial_{\beta_3}\partial_{\beta_1} \left(\log\det\Delta^S_{\pmb\beta}\Bigr|_{\beta_2=|\pmb\beta|-\beta_1-\beta_3}\right)\Bigr|_{\beta_j=\frac {|\pmb\beta|}{3}} 
\\
& =
- \frac 9 2 \left({|\pmb\beta|}+3\right)^{-3}(2\log 2+\log S +2\Phi\Bigr|_{\beta_j=\frac {|\pmb\beta|}{3}} )
+3\left( {|\pmb\beta|}+{3}\right)^{-2} (\partial_{\beta_1}\Phi-\partial_{\beta_2}\Phi)\Bigr|_{\beta_j=\frac {|\pmb\beta|}{3}}
\\
&\qquad+\frac 1 6\left( \frac {|\pmb\beta|+3}{3}-\frac 3 {|\pmb\beta|+3}\right)( \partial^2_{\beta_1}\Phi+\partial^2_{\beta_2}\Phi+\partial^2_{\beta_3}\Phi )\Bigr|_{\beta_j=\frac {|\pmb\beta|}{3}}-\mathcal C''\left(\frac {|\pmb\beta|}{3}\right).
\end{aligned}
$$
Here $\Phi$ is the  function found in~\eqref{Phi}, and the function $\mathcal C$ is defined in~\eqref{Cb}.
\end{proposition}
\begin{proof}  It is a bit tedious but straightforward  to derive  the explicit formulas for the second order derivatives. One  need only  use  the formulae~\eqref{rescaling},~\eqref{Zeta_0},  and~\eqref{DetDelta} for $\log\det\Delta^S_{\pmb\beta}$, the governing equations~\eqref{GovEqn} for the Liouville action, and the definition~\eqref{FH} of the functional $\mathcal H_{\pmb\beta}[\phi]$. We omit the details.
\end{proof}

\begin{corollary}\label{CSD} If the surface area $S$ is sufficiently small, then the stationary point in Proposition~\ref{StatPoints} is a minimum.

 More precisely,  for each $|\pmb\beta|\in(-3,0)$ there exists a number $S_0=S_0(|\pmb\beta|)$ such that for any $S\in(0,S_0]$ 
 the stationary point $\beta_1=\beta_2=\beta_3=\frac {|\pmb\beta|} 3$ is a minimum of the function~\eqref{DasFA}.
\end{corollary}
\begin{proof} From the formulas for the second order derivatives in Proposition~\ref{SecDer} it is easy to see that for each $|\pmb\beta|=\frac{SK}{2\pi}-2\in(-3,0)$ there exists $S_0=S_0(|\pmb\beta|)$ such that for any $S\in(0, S_0]$ we have
$$
\begin{aligned}
\partial^2_{\beta_1} \left(\log\det\Delta^S_{\pmb\beta}\Bigr|_{\beta_2=|\pmb\beta|-\beta_1-\beta_3}\right)\Bigr|_{\beta_j=\frac {|\pmb\beta|}{3}}
>\partial_{\beta_3}\partial_{\beta_1} \left(\log\det\Delta^S_{\pmb\beta}\Bigr|_{\beta_2=|\pmb\beta|-\beta_1-\beta_3}\right)\Bigr|_{\beta_j=\frac {|\pmb\beta|}{3}}
>0.
\end{aligned}
$$
These inequalities imply that the stationary point is a minimum. 
\end{proof}
\begin{proof}[Proof of Theorem~\ref{TSP}] The assertion is an immediate consequence of Proposition~\ref{StatPoints} and Corollary~\ref{CSD}.
\end{proof}

\begin{remark} In the  particularly simple case of the flat metrics (i.e. when $K=0$ and $|\pmb\beta|=-2$) it is relatively easy to clarify  what happens with the stationary point as the area increases: When the area $S$ is  below a certain value (approximately $1.92$),  the stationary point  $(\beta_1,\beta_2,\beta_3)$ with $\beta_j=-\frac 2 3$ is a minimum. When the area $S$  exceeds this value,  it is a maximum; cf.~\cite{KalvinJGA}. For instance, the Calabi-Croke sphere (or, equivalently, the unit side equilateral triangle envelope)  minimizes the determinant on the flat metrics (with three conical singularities) of  area $S=\sqrt{3}/2$.   
\end{remark}

\noindent{ \bf Acknowledgements} The author would like to thank Paul Wiegmann for numerous discussions.

\appendix
 \section{Appendix: Explicit solution to  Nirenberg problem}\label{NP}
 
  In this appendix  we  solve the  singular Nirenberg problem: We explicitly construct the conformal metric $e^{2\phi}|dz|^2$ with three conical singularities on the Riemann sphere,  given its Gaussian curvature $K=2\pi(|\pmb\beta|+2)$ and the orders $\beta_j\in(-1,0)$ of conical singularities.   In particular, we obtain explicit expressions for the coefficients $\phi_j$ in the asymptotics~\eqref{ASphi}  of the metric potential $\phi$ in terms of of the orders $\beta_j$ of conical singularities. We also to show that the coefficient $\phi_\infty$ in the asymptotics of $\phi$  at infinity is well-defined.

\subsection{Existence and uniqueness} \label{EandU}
\begin{lemma}\label{PotEU} Let $p_j\in\Bbb C$ be three distinct marked points. Assume that $\beta_j\in(-1,0)$ and $\beta_j-\frac{|\pmb\beta|}2>0$ for  $j=1,2,3$.  Then there exists a unique  solution $\phi$ to the Liouville equation
 \be\label{LAapp}
 e^{-2\phi}(-4\partial_z\partial_{\bar z}\phi)=2\pi(|\pmb\beta|+2),\quad z\in\Bbb C\setminus\{p_1,p_2,p_3\},
\ee
having the  asymptotics 
$$
\begin{aligned}
\phi(z) & =\beta_j\log|z-p_j|+\phi_j+o(1), \quad z\to p_j,
\\
 \phi(z)& =-2\log|z| +\phi_\infty+o(1), \quad z\to \infty,
\end{aligned}
$$
with some coefficients $\phi_j$ and $\phi_\infty$, and   satisfying  the unit area condition 
$$
\int_{\Bbb C} e^{2\phi}\frac {dz\wedge d\bar z }{-2i}=1.
$$  

In other words,    there exists a unique  unit area constant  curvature  conformal metric with three distinct conical singularities of order $\beta_j$.
\end{lemma}
\begin{proof} Recall that $2\pi(|\pmb\beta|+2)=K$ is the Gaussian curvature of the metric $e^{2\phi}|dz|^2$.
 Let us  consider the hyperbolic case ($K<0$), the flat case ($K=0$), and the spherical case ($K>0$) separately.

In the hyperbolic case we have $|\pmb\beta|<-2$ and  the inequalities  $\beta_j-|\pmb\beta|/2>0$ are a priori satisfied as $\beta_j\in(-1,0)$.  By the classical result of Picard~\cite{Picard}, there exists a unique conformal metric $e^{2\phi}|dz|^2$ of Gaussian curvature $K=2\pi(|
\pmb\beta|+2)<0$ representing the divisor $\pmb\beta=\sum \beta_j\cdot p_j$.   The unit area condition is equivalent to the equality $K=2\pi(|\pmb\beta|+2)$  as the Gauss-Bonnet theorem~\cite{Troyanov} reads $\int_{\Bbb C} K e^{2\phi}\frac {dz\wedge d\bar z }{-2i}=2\pi(|\pmb\beta|+2)$. 

 In the flat case we have $|\pmb\beta|=-2$. The inequalities  $\beta_j-|\pmb\beta|/2>0$ are automatically satisfied again. The assertion of lemma is a reformulation of results in~\cite[\S5]{TroyanovSSC}, where the unit area condition guarantees    uniqueness.

In the spherical case we have $|\pmb\beta|>-2$. This together with the inequalities  $\beta_j-|\pmb\beta|/2>0$ is  equivalent to the Troyanov condition
$$
0<|\pmb\beta|+2<2\min (\beta_j+1)
$$ 
that guarantees existence~\cite{Troyanov}. In general, the latter condition is a technical requirement needed for applicability of  Troyanov's method. However, for $\beta_j\in(-1,0)$ the  Troyanov condition is known to be necessary and sufficient for the existence  of  a conformal metric $e^{2\phi}|dz|^2$ of Gaussian curvature $K>0$ representing the divisor $\pmb\beta$, see~\cite{LT,UY,Eremenko}. Moreover, this metric is unique.  As in the hyperbolic case, the unit area condition is equivalent to the equality $K=2\pi(|\pmb\beta|+2)$    thanks to the Gauss-Bonnet theorem. 
\end{proof}

\subsection{Unit area singular metric in a closed explicit form}\label{UACCM}
By Lemma~\ref{PotEU}  there exists a unique  unit area constant  curvature  metric $e^{2\phi}|dz|^2$ with three distinct conical singularities of order $\beta_j\in(-1,0)$ satisfying $\beta_j-\frac{|\pmb\beta|}2>0$. 
In this subsection we construct  the metric $e^{2\phi}|dz|^2$ in a closed explicit form.  

We rely on classical  methods that  go back to the work of Riemann, Klein, Koebe, Schwarz, Pincar\'e, and Picard. The explicit form of constant curvature metrics with three  singularities  was known to  physicists for quite some time, e.g.~\cite{BG,Z-Z}. It was also studied by mathematicians, see e.g.~\cite{Kraus2011} and references therein. However,  we did not find in the literature any universal formula for the hyperbolic, flat, and spherical metrics of constant  curvature prescribed by the sum of the orders of conical singularities. So we deduce one here. 
 We also clarify some links with the Liouville field theory.

It is convenient to normalize the marked points $p_j$ so that $p_1=0$, $p_2=1$, and $p_3=\infty$. This can always be done by means of the M\"obius transformation 
$$
z\mapsto f(z)=\frac{p_2-p_3}{p_2-p_1}\cdot \frac{z-p_1}{z-p_3} .
$$
 Moreover, for the metric potential $\phi$ from Lemma~\ref{PotEU} we have 
$$
\phi(z)=\varphi\circ f(z)+\log|f'(z)|,
$$ 
where $\varphi$ stands for the potential of the unit area constant curvature metric $e^{2\varphi}|dz|^2$ representing the divisor
$$
 \pmb \beta= \beta_1\cdot 0+\beta_2\cdot 1+\beta_3\cdot \infty.
$$

The metric potential $\varphi$ can be found in the form 
\begin{equation}\label{MPST}
\varphi=\log\frac{2|w'|} {1+2\pi(|\pmb\beta|+2) |w|^2},
\end{equation}
where  the  developing map $w$ is  analytic in $\Bbb C\setminus\{0,1\}$, and $w'=\partial_z w$. For the 
 Schwarzian derivative $\{w,z\}=\frac{2 w'  w'''-3 w''^2}{2w'^2}$  we  obtain
$$
\{w,z\}=2 \left( \partial^2_z \varphi -(\partial_z\varphi)^2\right)=:T_\varphi(z),
$$
where $T_\varphi$ is the classical stress-energy tensor. As a consequence of the Liouville equation~\eqref{LAapp} we get
$$\begin{aligned}
0=\partial_{ z}\left( \partial_z\partial_{\bar z} \varphi +\frac K  4 e^{2\varphi }\right)=\frac 1 2 \partial_{\bar z} T_\varphi, \quad z\in\Bbb C\setminus\{0,1\}.
\end{aligned}
$$
The following expression and the asymptotics at infinity are due to Schwarz:
$$
T_\varphi(z)=\frac {\delta_1}{2z^2}+\frac {h_1}{z}+\frac {\delta_2}{2(1-z^2)}+\frac {h_2}{1-z},\quad T_\varphi(z)=\frac {\delta_3}{2z^2}+\frac{h_3}{z^3}+O(z^{-4})\quad\text{as}\quad z\to\infty,
$$
see e.g.~\cite{Car}. Here
$$
h_1=h_2=\frac{\delta_1+\delta_2-\delta_3}{2},
\quad
h_3=\frac{\delta_2+\delta_3-\delta_1}{2}
$$
are the accessory parameters and  $\delta_j=-\beta_j(2+\beta_j)$. 

Consider the hypergeometric differential equation 
\begin{equation}\label{HypEqn}
z(1-z)u''+(-\beta_1-(\beta_3-1-|\pmb\beta|)z)u'-  (\beta_3-|\pmb \beta|/2)(-1-|\pmb \beta|/2)u=0.
\end{equation}
It so happens that  the quotient $w=u_2/u_1$ of any two linearly independent solutions $u_1$ and $u_2$ to the hypergeometric equation satisfies the equation $\{w,z\}=T_\varphi(z)$. 

For instance, one can take 
$$
u_1(z)=F(\beta_3-|\pmb \beta|/2,-1-|\pmb \beta|/2,-\beta_1; z),
$$
$$
u_2(z)=c_{\pmb\beta} z^{\beta_1+1} F(1-\beta_2+|\pmb \beta|/2,\beta_1-|\pmb \beta|/2,2+\beta_1; z),
$$
where $F(a,b,c;z)$ stands for the hypergeometric function and $c_{\pmb\beta}$ is a scaling factor.
As a result we arrive at  the Schwarz triangle function 
\begin{equation}\label{STF}
w( z)=c_{\pmb\beta} z^{\beta_1+1} \frac{F(1-\beta_2+|\pmb \beta|/2,\beta_1-|\pmb \beta|/2,2+\beta_1;  z)}{F(\beta_3-|\pmb \beta|/2,-1-|\pmb \beta|/2,-\beta_1;z)}.
\end{equation}
This function satisfies  $w(0)=0$. As we show in the proof  of  Proposition~\ref{METRIC} below,  by setting
\be\label{CSF}
c_{\pmb\beta}=\frac {\exp\{\Phi(\beta_1,\beta_2,\beta_3)\}}{\beta_1+1}
\ee
with the function  $\Phi$  defined in~\eqref{Phi}, we normalize the Schwarz triangle function $w( z)$  so that it maps the upper half-plane $\Im z>0$ to a geodesic triangle $OAB$ in  the model metric 
\begin{equation}\label{ModelMetric}
\frac {4|d w|^2}{(1+2\pi(|\pmb\beta|+2)|w|^2)^{2}}
 \end{equation}
of  Gaussian  curvature $K=2\pi(|\pmb\beta|+2)$, cf. Fig.~\ref{TemplateHyp}, Fig.~\ref{TemplateSph}, and Fig.~\ref{Template}. Or, equivalently,  so that the metric potential in~\eqref{MPST} 
 is a real single-valued function on $\overline{\Bbb C}$, see~e.g.~\cite{BG,CMS,Eremenko,HJ,Kraus2011,LT,T-Z,UY}.  
 
 In the context of the  Liouville quantum field theory, the metric potential $\varphi$ is the field,  the coefficients $\delta_j$ are the conformal dimensions or weights, $T_\varphi$ is the $(2,0)$-component of the stress-energy tensor, 
 see  e.g.~\cite{BPZ,C-W,CMS,DO,HJ,T-Z,Z-Z}.

 \begin{figure}[h]
\centering\begin{tikzpicture}[scale=1.8]
\draw[gray!50,fill=gray!25]  (6.75,0)--(0,0) -- (5,2) arc (-165:-112:3cm);
\draw[gray!50,fill=gray!50]  (0,0) -- (5,-2) arc (165:112:3cm);
\draw[black,solid,thick] (6.75,0)node[anchor=west]{$A$}--(0,0)node[anchor=north east]{$O$} -- (5,2)node[anchor=south]{$B$} arc (-165:-112:3cm) ; 
\draw[black,solid,thick]  (0,0) -- (5,-2)node[anchor=north]{$B'$} arc (165:112:3cm); 

\filldraw[black] (0,0) circle (1.5pt);\filldraw[black] (6.75,0) circle (1.5pt);\filldraw[black] (5,2) circle (1.5pt);\filldraw[black] (5,-2) circle (1.5pt);

\draw[black] (1.5,0) arc (0:33:1cm);
\draw[black] (2,0.15)node[anchor=south]{$\pi(\beta_1+1)$};

\draw[black] (5.2,1.48) arc (-73:-146:.7cm);
\draw[black] (4.7,1.48)node[anchor=north]{$\pi(\beta_3+1)$};

\draw[black] (5.5,0) arc (180:138:1cm);
\draw[black] (5,.15)node[anchor=south]{$\pi(\beta_2+1)$};
\end{tikzpicture}
\caption{   Hyperbolic  geodesic triangle  $OAB$ with internal angles $\pi(\beta_j+1)$,  $|\pmb\beta|<-2$, and its reflection $OAB'$ in the side $OA$. }
\label{TemplateHyp}
\end{figure}
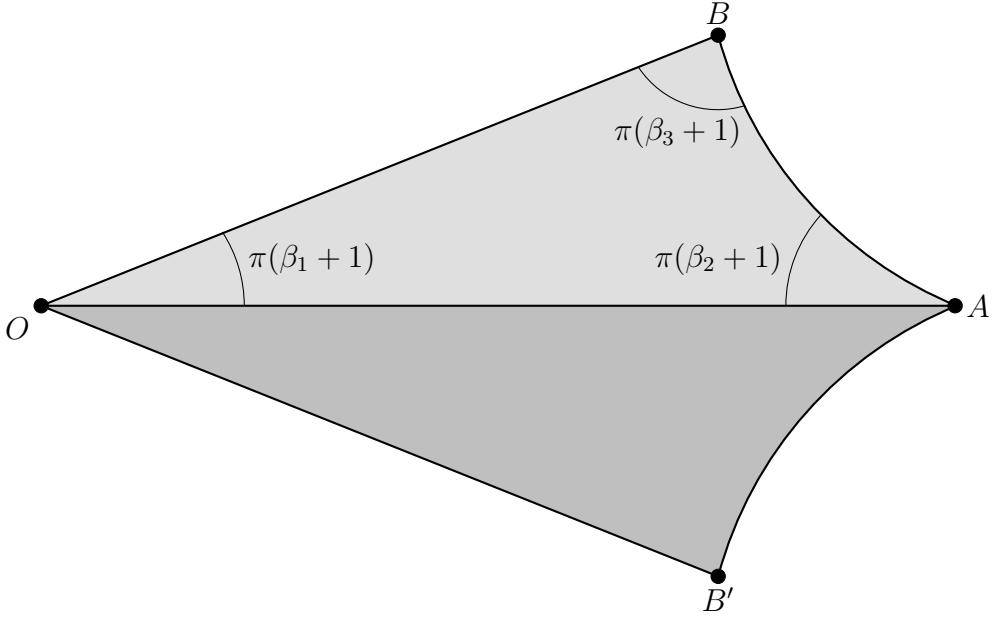

\begin{figure}[h]
\centering\begin{tikzpicture}[scale=1.8]
\draw[gray!50,fill=gray!25]  (4.98,2.01)--(0,0) --(6.75,0)arc (180-165:180-112:3cm) ;
\draw[gray!50,fill=gray!50]  (4.98,-2.01)--(0,0) --(6.75,0)arc (-180+165:-180+112:3cm) ;
\draw[black,solid,thick] (4.98,2.01)node[anchor=south]{$B$}--(0,0)node[anchor=north east]{$O$} --(6.75,0) node[anchor=west]{$A$}arc (180-165:180-112:3cm)   ; 
\draw[black,solid,thick]  (4.98,-2.01)node[anchor=north]{$B'$}--(0,0) --(6.75,0)arc (-180+165:-180+112:3cm);

\filldraw[black] (0,0) circle (1.5pt);\filldraw[black] (6.75,0) circle (1.5pt);\filldraw[black] (5,2) circle (1.5pt);\filldraw[black] (5,-2) circle (1.5pt);

\draw[black] (1.5,0) arc (0:33:1cm);
\draw[black] (2,0.15)node[anchor=south]{$\pi(\beta_1+1)$};

\draw[black] (5.55,1.7) arc (-50:-135:.9cm);
\draw[black] (4.9,1.55)node[anchor=north]{$\pi(\beta_3+1)$};

\draw[black] (5.9,0) arc (180:122:1cm);
\draw[black] (5.5,.3)node[anchor=south]{$\pi(\beta_2+1)$};
\end{tikzpicture}
\caption{Spherical  geodesic triangle  $OAB$ with internal angles $\pi(\beta_j+1)$,  $|\pmb\beta|>-2$, and its reflection $OAB'$ in the side $OA$.}
\label{TemplateSph}
\end{figure}
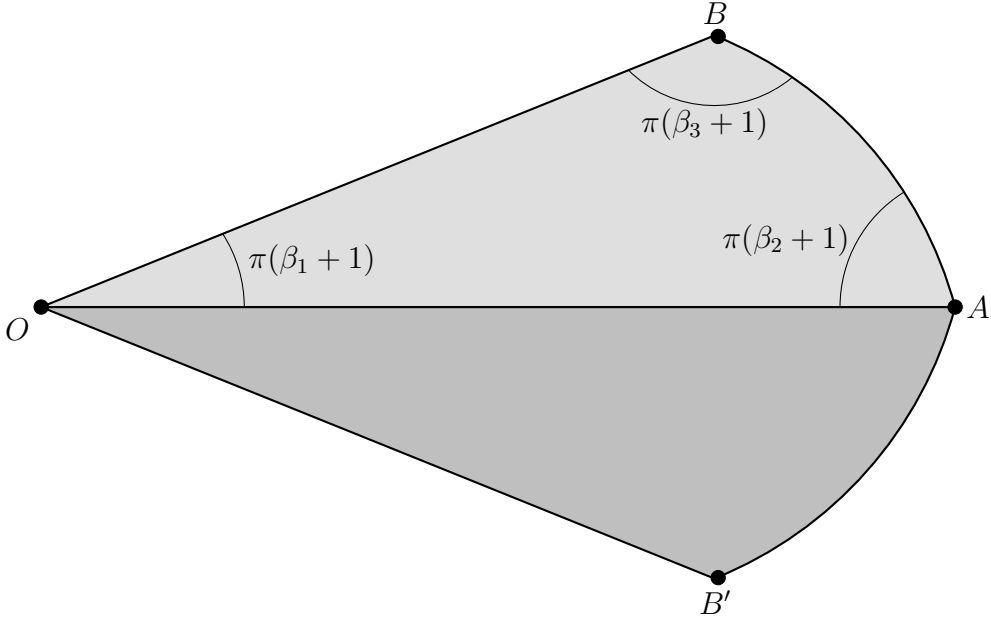

\begin{figure}[h]
\centering\begin{tikzpicture}[scale=1.8]
\draw[gray!50,fill=gray!25] (0,0) -- (5,2)--(7,0)--(0,0);
\draw[gray!50,fill=gray!50] (0,0) -- (7 ,0) -- (5,-2)--(0,0); 
\draw[black,solid,thick] (0,0)node[anchor=north east]{$O$} -- (5,2)node[anchor=south]{$B$}--(7,0)node[anchor=west]{$A$}--(0,0); 
\draw[black,solid, thick]  (7 ,0) -- (5,-2)node[anchor=north]{$B'$}--(0,0); 

\filldraw[black] (0,0) circle (1.5pt);\filldraw[black] (7,0) circle (1.5pt);\filldraw[black] (5,2) circle (1.5pt);\filldraw[black] (5,-2) circle (1.5pt);

\draw[black] (1.5,0) arc (0:33:1cm);
\draw[black] (2,0.15)node[anchor=south]{$\pi(\beta_1+1)$};

\draw[black] (5.5,1.5) arc (-60:-138:1cm);
\draw[black] (4.9,1.4)node[anchor=north]{$\pi(\beta_3+1)$};

\draw[black] (6,0) arc (180:135:1cm);
\draw[black] (5.6,.25)node[anchor=south]{$\pi(\beta_2+1)$};
\end{tikzpicture}
\caption{Euclidean triangle  $OAB$ with internal angles $\pi(\beta_j+1)$,  $|\pmb\beta|=-2$, and its reflection $OAB'$ in the side $OA$.} \label{Template}
\end{figure}
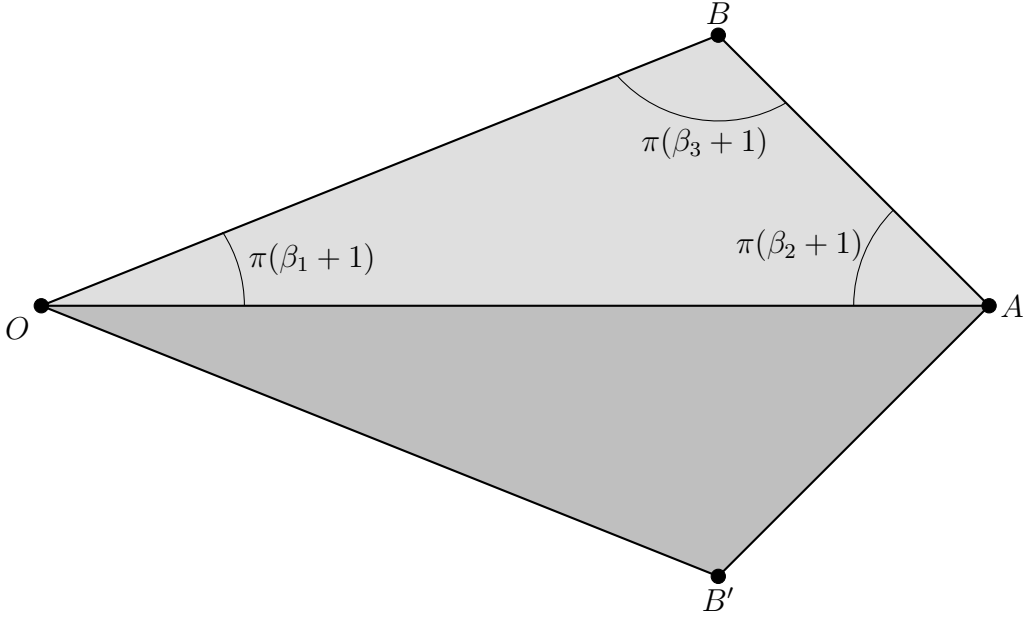

 The Schwarz triangle function $w( z)$ in~\eqref{STF}  maps the point $z=0$ to the origin  $O$,   the point $z=1$ to the vertex $A$, and the upper half-plane $\Im z>0$ to 
 \begin{itemize}
 \item  The hyperbolic geodesic triangle $OAB$ in Fig.~\ref{TemplateHyp}, if $|\pmb\beta|<-2$;
 
 \item The spherical geodesic triangle $OAB$ in Fig.~\ref{TemplateSph}, if $|\pmb\beta|>-2$; 
 
 \item The Euclidean triangle $OAB$ in Fig.~\ref{Template}, if  $|\pmb\beta|=-2$.
 \end{itemize}
 The analytic continuation of $ z\mapsto w( z)$ (from the upper half-plane $\Im z>0$ through the interval $(0,1)$ of the real axis) maps the lower half-plane $\Im  z<0$ into the reflection $OAB'$ of the geodesic triangle $OAB$ in the side $OA$. 

We use the Schwarz triangle function~\eqref{STF} as the developing map for the model metric~\eqref{ModelMetric}, i.e. we introduce  the metric $e^{2\varphi}|d z|^2$ as the pullback of the  model metric~\eqref{ModelMetric} by the Schwarz triangle function $w( z)$. Or, equivalently, we find the metric potential $\varphi$ in the form~\eqref{MPST}. Thus we  obtain the unit area,  Gaussian curvature $K=2\pi(|\pmb \beta|+2)$ conformal metric 
\begin{equation}\label{01inf}
e^{2\varphi}|dz|^2= \frac{4|w'(  z)|^2\,|d z|^2}{(1+2\pi(|\pmb \beta|+2)|w( z)|^2)^2}
\end{equation}
 representing the divisor
$$
 \pmb \beta= \beta_1\cdot 0+\beta_2\cdot 1+\beta_3\cdot \infty.
$$
By Lemma~\ref{PotEU} there exists exactly one metric with these properties. Thus the metric~\eqref{01inf} is the explicit solution to the singular Nirenberg problem, see also Remark~\ref{EFphi} at the end of this section.

In the case $|\pmb\beta|<-2$  (resp. $|\pmb\beta|>-2$) the unit area surface $(\overline{\Bbb C}, e^{2\varphi}|dz|^2)$  can be visualized as  a hyperbolic  (resp. spherical) geodesic triangle with internal angles $\pi(\beta_j+1)$ glued along the edges to its reflection in a side; see Fig.~\ref{TemplateHyp}  and Fig.~\ref{TemplateSph}, where the geodesic triangles are already glued along $OA$, the resulting geodesic quadrilateral needs to be folded along $OA$, then $OB$  should be glued to $OB'$ and $OA$ to $OA'$.

 In the  case   $|\pmb \beta|=-2$ the model metric~\eqref{ModelMetric} is flat. The Schwarz triangle function reduces to the  Schwarz-Christoffel transformation
\be\label{S-C}
w( z)\bigr|_{|\pmb\beta|=-2}
=\exp\{\Phi(\beta_1,\beta_2,\beta_3)\}\int_0^{  z}   z^{\beta_1} (1-z)^{\beta_2}\,d z.
\ee
For the pullback of the model  metric~\eqref{ModelMetric} with $|\pmb\beta|=-2$ by $w( z)\bigr|_{|\pmb\beta|=-2}$ we immediately obtain
\be\label{fm}
e^{2\varphi}|dz|^2\bigr|_{|\pmb\beta|=-2}= 4\exp\{2\Phi(\beta_1,\beta_2,\beta_3)\} | z|^{2\beta_1}| z-1|^{2\beta_2}|d z|^2.
\ee
The  surface $(\overline{\Bbb C}, e^{2\varphi}|dz|^2\bigr|_{|\pmb\beta|=-2})$ can be visualized as a flat triangle envelope: a Euclidean triangle (of area $1/2$) glued along the edges to its reflection in a side, cf.~Fig.~\ref{Template}.

In other words, we  explicitly constructed the following uniformization of the unit area constant curvature genus zero surfaces with three conical singularities: the Riemann sphere $\overline{\Bbb C}$  equipped with the singular  metric $e^{2\varphi}|dz|^2$ in~\eqref{01inf} is isometric to a constant curvature surface glued from a  Hyperbolic,  Spherical, or Euclidean geodesic triangle and its reflection in a side. The isometry is given by the Schwarz triangle function~\eqref{STF}.

Up to now for the three distinct marked points $p_j$ of the Riemann sphere $\overline{\Bbb C}$ we were using the normalization $p_1=0$, $p_2=1$, and $p_3=\infty$. Now we apply the M\"obius transformation $ \hat  z=\frac{1+z}{1-z}$ and pass to the normalization $p_1=-1$, $p_2=0$, and $p_3=1$ that we use throughout the paper.

\begin{proposition}\label{METRIC}  Assume that $\beta_j\in(-1,0)$ and $\beta_j-\frac{|\pmb\beta|}2>0$ for all $j=1,2,3$.   Then  there exists a unique unit area  constant curvature
conformal metric $e^{2\phi}|dz|^2$   representing the divisor
\be\label{DIV}
\pmb\beta=\beta_1\cdot (-1)+\beta_2\cdot 0+\beta_3\cdot 1.
\ee
The  potential $\phi$ of this metric  satisfies the estimates 
\begin{equation}\label{est_phi}
\begin{aligned}
\phi(z) & =\beta_1\log|z+1|+\phi_1+o(1), \quad z\to -1,
\\
\phi(z) & =\beta_2\log|z|+\phi_2+o(1), \quad z\to 0,
\\
\phi(z) & =\beta_3  \log|z-1| +\phi_3+o(1), \quad z\to 1,
\end{aligned}
\end{equation}
with the  coefficients
\begin{equation}\label{phi123}
\begin{aligned}
\phi_1  & = -\beta_1\log 2 +  \Phi(\beta_1,\beta_2,\beta_3),
\\
\phi_2& =(\beta_2+2)\log 2 + \Phi(\beta_2,\beta_1,\beta_3),
\\
\phi_3 & =  - \beta_3\log 2   +\Phi(\beta_3,\beta_2,\beta_1).
\end{aligned}
\end{equation}
Here
\begin{equation}\label{Phi}
\begin{aligned}
\Phi(\beta_1,\beta_2,\beta_3)=&  \frac 1 2 \log   \frac{\Gamma\left(2+{|\pmb \beta|}/2\right)}{  4\pi \Gamma\left(-{|\pmb \beta|} /2\right) }
+ \log \frac{\Gamma(-\beta_1)} {\Gamma(1+\beta_1)}
\\
 &+\frac 1 2 \log \frac{\Gamma(\beta_1-|\pmb \beta|/2)
 \Gamma(1+|\pmb \beta|/2-\beta_2)  \Gamma(1+|\pmb \beta|/2-\beta_3) }{   \Gamma(1+|\pmb \beta|/2-\beta_1)    \Gamma(\beta_2-|\pmb \beta|/2) \Gamma(\beta_3-|\pmb \beta|/2) }.
 \end{aligned}
\end{equation}
In addition, as $z\to\infty$  the metric potential $\phi$ meets the estimate
\begin{equation}\label{inf_est}
\phi(z)  =-2  \log|z| +\phi_{\infty}+o(1)
\end{equation}
with
\begin{equation}\label{phi_infty}
\phi_\infty=2\log 2 +\log|w'(-1)| -\log (1+2\pi(2+|\pmb \beta|)|w(-1)|^2),
\end{equation}
where $w(z)$  is  the Schwarz triangle function  defined via~\eqref{STF},~\eqref{SF}, and ~\eqref{s2}.
\end{proposition}
 Note that  the arguments of all gamma functions in~\eqref{Phi} stay positive. Indeed, for $\beta_j\in(-1,0)$ the expressions $1+|\pmb \beta|/2-\beta_j$ are always positive.
In the hyperbolic case we have $|\pmb\beta|<-2$, and hence the inequalities  $\beta_j-|\pmb\beta|/2>0$ hold true. In the spherical case we have $|\pmb\beta|>-2$, and  the conditions   $\beta_j-|\pmb\beta|/2>0$  are necessary and sufficient for the existence of the metric with three conical singularities of order $\beta_j\in(-1,0)$.
 In the flat case the equality~\eqref{Phi} takes the form 
\be\label{fmPhi}
\Phi(\beta_1,\beta_2,\beta_3)\bigr|_{|\pmb\beta|=-2}=\frac 1 2 \log \frac  {\Gamma(-\beta_1) \Gamma(-\beta_2) \Gamma(-\beta_3)}  {4\pi \Gamma(\beta_1+1)\Gamma(\beta_2+1)\Gamma(\beta_3+1)}. 
\ee

\begin{proof}[Proof of Proposition~\ref{METRIC}]

 By setting
\begin{equation}\label{SF}
c_{\pmb\beta}=s\frac{\Gamma(-\beta_1)\Gamma(1+|\pmb \beta|/2- \beta_3)\Gamma(2+|\pmb \beta|/2)}{\Gamma(2+\beta_1)\Gamma(\beta_2-|\pmb \beta|/2)\Gamma(1+|\pmb \beta|/2-\beta_1)}
\end{equation}
in~\eqref{STF} we achieve    $w(1)=s$, see e.g.~\cite[Vol.\,2,\S392]{Car}. Now we need to find $s>0$   so that  $w( z)$ maps the upper half-plane $\Im z>0$ to a geodesic triangle in  the model metric~\eqref{ModelMetric}.

 In the case $|\pmb\beta|<-2$ the negative curvature model metric~\eqref{ModelMetric} can be reduced to the standard Gaussian curvature $-1$ metric  
 $
 {4(1-| z^2|)^{- 2}|d z|^2}{}
 $ 
 in the Poincar\'e disk $| z|<1$ by the substitution $ w=(-2\pi(|\pmb\beta|+2))^{-1/2} z$    (with subsequent multiplication of the resulting metric by  $-2\pi(|\pmb\beta|+2)$, which does not affect the shape of geodesics). The expression for  $s$, which guarantees that the triangle is a geodesic triangle in the standard Poincar\'e disk (of Gaussian curvature $-1$), is well-known (see e.g.~\cite[Vol.\,2, eqn. (392.4)]{Car}). In order to make the triangle geodesic with respect to the metric~\eqref{ModelMetric}, we only need to multiply that known expression  by $(-2\pi(|\pmb\beta|+2))^{-1/2}$. As a result,  for $s$ in~\eqref{SF} we obtain
\begin{equation}\label{s2}
 s^2=\frac{  \Gamma(\beta_1 -|\pmb \beta|/2)\Gamma(\beta_2-|\pmb \beta|/2)\Gamma(1+|\pmb \beta|/2 -\beta_1)\Gamma(1+|\pmb \beta|/2-\beta_2)}{4\pi\Gamma(-|\pmb \beta|/2)\Gamma(2+|\pmb \beta|/2) \Gamma(\beta_3-|\pmb \beta|/2) \Gamma(1-\beta_3+|\pmb \beta|/2)}.
\end{equation}

Similarly, in the case $|\pmb\beta|>-2$ the metric ~\eqref{ModelMetric} can be reduced to the standard  spherical (Gaussian curvature one) metric 
$ {4 (1+| z|^2)^{-2}| d z|^2}{}$ 
by the substitution $ w=(2\pi(|\pmb\beta|+2))^{-1/2} z$    (again with subsequent multiplication of the resulting metric by  $2\pi(|\pmb\beta|+2)$, which does not affect the shape of geodesics). For the standard spherical metric the expression for $s$, that guarantees that the triangle is a geodesic spherical triangle, is also well known, see e.g.~\cite[Vol.\,1,  eqn.~(64.11) and (72.5)]{Car}. Multiplying that known expression by $(2\pi(|\pmb\beta|+2))^{-1/2}$, we come to exactly the same equality~\eqref{s2} as  above. Thus in the case $|\pmb\beta|>-2$ the function $w( z)$ maps the upper half-plane $\Im z>0$ to a geodesic spherical triangle in  the metric~\eqref{ModelMetric}.

Now let us notice that the equality~\eqref{Phi} is equivalent to 
$$
\Phi(\beta_1,\beta_2,\beta_3):=\log (\beta_1+1)+\log c_{\pmb\beta}
$$
with $c_{\pmb\beta}$ defined by~\eqref{SF} and~\eqref{s2}. 

In the flat case $|\pmb\beta|=-2$ the Schwarz triangle function reduces to the  Schwarz-Christoffel transformation~\eqref{S-C} that maps the upper half-plane  to a Euclidean triangle.

We have demonstrated that  the choice of the scaling factor $c_{\pmb\beta}$ in~\eqref{CSF} is correct, i.e. the Schwarz triangle function~\eqref{STF} maps the upper half-plane $\Im z>0$ to a geodesic triangle in the model metric~\eqref{ModelMetric}. Or, equivalently, that the metric potential~\eqref{MPST}  is a real single-valued function on $\overline{\Bbb C}$. This justifies the construction of  the unit area Gaussian curvature $K=2\pi(|\pmb \beta|+2)$ conformal metric $e^{2\varphi}|dz|^2$ in~\eqref{01inf}.
 (It is a unit area metric as it follows from the Gauss-Bonnet theorem for $|\pmb\beta|\neq -2$, and either from a direct verification or by continuity for $|\pmb\beta|=-2$.)   By Lemma~\ref{PotEU} this metric is unique. 

At $ z=0$ the Schwarz triangle function $w( z)$ in~\eqref{STF} and its derivative
\begin{equation}\label{w'}
w'( z)=\frac{ c_{\pmb\beta}(\beta_1+1){ z}^{\beta_1}  (1- z)^{\beta_2}}{ (F( \beta_3-|\pmb \beta|/2,-1-|\pmb \beta|/2,-\beta_1; z))^2}
\end{equation}
 take the values
$$
w(0)=0,\quad  z^{-\beta_1} w'( z)|_{ z=0}=c_{\pmb\beta}(\beta_1+1)=\exp
\{\Phi(\beta_1,\beta_2,\beta_3)\}.
$$
Hence for the potential  of the metric $e^{2\varphi}|dz|^2$ in~\eqref{01inf} 
we have
$$
\varphi( z)=\beta_1\log| z |+\log 2+\Phi(\beta_1,\beta_2,\beta_3)+o(1), \quad z\to 0.
$$
The potential $\varphi ( z)$ is related to  the potential $\phi (z)$ of the metric~ $e^{2\phi}|dz|^2$ (corresponding to the normalization  $p_1=-1$, $p_2=0$, and $p_3=1$ of the marked points)
 by the equality
\begin{equation}\label{10:26}
\phi (z)= \varphi \circ f(z) +\log |\partial_z f(z)|   ,
\end{equation}
where $f$  is the M\"obius transformation $z\mapsto f(z)=\frac{1+z}{1-z}$. This implies the first estimate in~\eqref{est_phi} with the coefficient $\phi_1$ given in~\eqref{phi123}.
 
 Similarly, (by interchanging the roles of $\beta_1$ and $\beta_2$) for the potential $\varphi$ of a unit area, Gaussian  curvature $2\pi(|\pmb\beta|+2)$ metric representing the divisor
 $$ 
\pmb\beta= \beta_2\cdot 0+\beta_1\cdot 1+\beta_3\cdot\infty
$$ 
we obtain
$$
\varphi ( z)=\beta_2\log| z |+\log 2+\Phi(\beta_2,\beta_1,\beta_3)+o(1), \quad z\to 0.
$$
This potential $\varphi ( z)$ is related to  the potential $\phi (z)$ by the equality~\eqref{10:26}, where $f$  is the M\"obius transformation $z\mapsto f(z)=\frac{2z}{1-z}$.
This establishes the second estimate in~\eqref{est_phi} with $\phi_2$ given in~\eqref{phi123}.

Finally,  the potential of a unit area, curvature $2\pi(|\pmb\beta|+2)$ metric representing the divisor
 $$
\pmb\beta= \beta_3\cdot 0+\beta_2\cdot 1+\beta_1\cdot\infty
$$ 
satisfies
$$
\varphi(z)=\beta_3\log|z |+\log 2+\Phi(\beta_3,\beta_2,\beta_1)+o(1), \quad z\to 0.
$$
This potential $\varphi( z)$ is related to  the potential $\phi(z)$ by the equality~\eqref{10:26} with the M\"obius transformation $z\mapsto f(z)=\frac{1-z}{1+z}$.
This justifies the third estimate in~\eqref{est_phi}  with $\phi_3$ given in~\eqref{phi123}.

It remains to clarify the behaviour of $\phi(z)$ as $z\to\infty$. With this aim in mind we notice that
$$
w(-1)
=  2^{-1-\beta_1} \frac {\exp\{\Phi(\beta_1,\beta_2,\beta_3)\}}{\beta_1+1} \frac{F(1-\beta_2+|\pmb \beta|/2,2+|\pmb \beta|/2,2+\beta_1; 1/2)}{F(\beta_3-|\pmb \beta|/2,1-\beta_1+|\pmb \beta|/2,-\beta_1;1/2)},
$$
$$
|w'(-1)|
=\frac{ 2^{\beta_3-\beta_1}\exp\{\Phi(\beta_1,\beta_2,\beta_3)\}}{  (F( \beta_3-|\pmb \beta|/2,1-\beta_1+|\pmb \beta|/2,-\beta_1;1/2))^2};
$$
for the corresponding property of hypergeometric functions see e.g.~\cite[15.3.5]{AbSt}.
At $z=1/2$ the series defining the hypergeometric functions are rapidly convergent. Thus $\phi_\infty $ in~\eqref{phi_infty}  is well defined. For the potential $\varphi( z )$ of the metric~\eqref{01inf} we obtain
$$
\varphi(z)=\log 2 +\log|w'(-1)| -\log (1+2\pi(2+|\pmb \beta|)|w(-1)|^2)+o(1), \quad  z\to 0.
$$
Thanks to the equaltity~\eqref{10:26},  where $f(z)=\frac{1+z}{1-z}$, we arrive at the estimate~\eqref{inf_est}. This completes the proof. 
  \end{proof}
  
  \begin{remark}\label{EFphi}
For the potential $\phi$  of the constant curvature unit area metric  $e^{2\phi}|dz|^2$ representing the divisor $\pmb\beta=\beta_1\cdot(-1)+\beta_2\cdot 0+\beta_3\cdot 1 $ one can write out  the explicit closed formula
\begin{equation}\label{EEphi}
\begin{aligned}
\phi(z)=-2\log\left(     \frac  {1}{2^{\beta_2+2} (\beta_1+1)c_{\pmb\beta}} \left| \psi_1(z;\pmb\beta)\right|^2+ \frac{ 2^{\beta_2+1}  \pi(|\pmb\beta|+2)c_{\pmb\beta} } {\beta_1+1} \left| \psi_2(z;\pmb\beta)\right|^2  \right),
\end{aligned}
\end{equation}
where $c_{\pmb\beta}$ is the scaling factor~\eqref{CSF} with the function $\Phi$ defined in~\eqref{Phi}. The functions $\psi_1$ and $\psi_2$ in~\eqref{EEphi} are given by the equalities
\begin{equation*}
\begin{aligned}
\psi_1(z;\pmb\beta)  = (  {1+z} )^{-\beta_1/2} z^{-\beta_2/2} & (1-z)^{1+|\pmb\beta|/2-\beta_3/2}
\\
 &\times F\left(\beta_3-\frac{|\pmb \beta|}2,-1-\frac{|\pmb \beta|}2,-\beta_1;\frac{1+z}{1-z}\right),
\\
\psi_2(z;\pmb\beta)  = ( {1+z} )^{\beta_1/2+1}  z^{1+\beta_2/2}  & (1-z)^{-1-|\pmb\beta|/2+\beta_3/2}
\\
&\times F\left(1-\beta_3+\frac{|\pmb\beta|} 2,2+\frac{|\pmb\beta|}{2},2+\beta_1;\frac{1+z}{1-z}\right),
\end{aligned}
\end{equation*}
where $F(a,b,c;z)$ stands for the hypergeometric function; cf.~\cite[eqns.~(4.2),~(4.4), and~(4.5), where $\eta_j=-\beta_j/2$]{Z-Z}. Indeed, the M\"obius transformation $z\mapsto \frac {1+z}{1-z}$  brings the explicitly constructed metric~\eqref{01inf}  into the form $e^{2\phi}|dz|^2$ with $\phi$ given in~\eqref{EEphi}.
\end{remark}


\begin{thebibliography}{100}
\bibitem{AbSt} M. Abramowitz, I. Stegun, 
Handbook of mathematical functions with formulas, graphs, and mathematical tables.
 National Bureau of Standards Applied Mathematics Series, No. 55 U. S. Government Printing Office, Washington, D.C., 1964

\bibitem{AKR} C. Aldana, K. Kirsten, J. Rowlett, Polyakov formulas for conical singularities in two dimensions, Preprint 2020, arXiv:2010.02776v2

\bibitem{Alvarez} O. Alvarez, Theory of strings with boundary, Nucl. Phys. B 216 (1983), 125--184

\bibitem{Au-Sal} E. Aurell,  P. Salomonson,  On functional determinants of Laplacians in polygons and simplicial complexes. Comm. Math. Phys. 165 (1994), no. 2, 233--259.

\bibitem{AS2} E. Aurell,  P. Salomonson, Further results on Functional Determinants of Laplacians in Simplicial Complexes, Preprint May 1994, arXiv:hep-th/9405140


\bibitem{BPZ} A. Belavin,  M. Polyakov,  A.  Zamolodchikov, Infinite conformal symmetry in two-dimensional quantum field theory, Nucl. Phys B241 (1984) 333

\bibitem{BG} A. Bilal, J-L. Gervais, Construction of constant curvature punctured Riemann surfaces with particle-scattering interpretation, J. Geom. Phys. 5 (1988), 277--304



\bibitem{BFK} D. Burghelea D., L. Friedlander, and T. Kappeler, Meyer-Vietoris type formula for determinants of elliptic differential operators, J. Funct. Anal. 107 (1992), 34--65

\bibitem{C-W} T. Can, P. Wiegmann,  Quantum Hall states and conformal field theory on a singular surface, J. Phys. A50 (2017), 494003, arXiv:1709.04397

\bibitem{CMS} L. Cantini, P. Menotti, D. Seminara, Proof of Polyakov conjecture for general elliptic singularities, Phys. Lett. B 517 (2001), arXiv:hep-th/0105081


\bibitem{Car} C. Caratheodory, Theory of functions of a complex variable, Vol. 1 \& 2, Chelsea Publishing Company, New York, 1954. 



\bibitem{DO} H. Dorn, H.-J. Otto, Two and three point functions in Liouville theory, Nuclear Physics B 429 (1994), 375--388


 \bibitem{EW} J. Edward, S. Wu, Determinant of the Neumann operator on smooth Jordan curves, Proc. AMS 111 (1991), 357--363
  

\bibitem{Eremenko} A. Eremenko, Metrics of positive curvature with conic singularities on the sphere, Proc. Amer. Math. Soc. 132 (2004), 3349--3355

\bibitem{pgamma} O.  Espinosa,  V. Moll, A Generalized polygamma function. Integral Transforms and Special Functions (2004), 101--115, 	arXiv:math/0305079

  
 \bibitem{GG} C. Guillarmou, L. Guillop\'e,  The determinant of the Dirichlet-to-Neumann map for surfaces with boundary,  International Mathematics Research Notices  2007 (2007)

\bibitem{HJ} L. Hadasz, Z. Jask\'olski, Liouville theory and uniformization of four-punctured sphere, J. Math. Phys. 47, 082304 (2006)

\bibitem{H1} L. Hillairet, Formule de trace sur une surface euclidienne \`a singularit\'es coniques.  C. R. Math. Acad. Sci. Paris 335 (2002), 1047--1052.

\bibitem{H2} L. Hillairet, Contribution of periodic diffractive geodesics. J. Funct. Anal. 226 (2005), 48--89.

\bibitem{H3} L. Hillairet, Diffractive geodesics of a polygonal billiard. Proc. Edinb. Math. Soc. 49 (2006), 71--86.



\bibitem{HK} L. Hillairet, A. Kokotov, Krein formula and $S$-matrix for Euclidean surfaces with conical singularities. J. Geom. Anal. 23 (2013),  1498--1529. 


\bibitem{KalvinJFA} V. Kalvin, Polyakov-Alvarez type comparison formulas for determinants of Laplacians on Riemann surfaces with conical singularities, J. Funct. Anal. 280 (2021), no. 7, Paper No. 108866, arXiv:1910.00104

\bibitem{KalvinJGA} V. Kalvin, Spectral determinant on Euclidean isosceles triangle envelopes, J. Geom. Anal. 31 (2021), 12347--12374,	arXiv:2010.02209

\bibitem{KalvinCCM} V. Kalvin, Determinant of Friedrichs Dirichlet Laplacians on $2$-dimensional hyperbolic cones,  to appear in Commun. Contemp. Math., DOI: 10.1142/S0219199721501078 , arXiv:2011.05407 

 \bibitem{KalvinJGA19} V. Kalvin, On determinants of Laplacians on compact Riemann surfaces equipped with pullbacks of conical metrics by meromorphic functions.  J. Geom. Anal. 29 (2019), pp. 785--798, arXiv:1712.05405


\bibitem{IMRN} V. Kalvin, A. Kokotov, Metrics of constant positive curvature with conical singularities, Hurwitz spaces, and determinants of Laplacians. Int. Math. Res. Not. IMRN no. 10 (2019), pp. 3242--3264, arXiv:1612.08660
 
 \bibitem{Bul} V. Kalvin, A. Kokotov, Determinant of the Laplacian on tori of constant positive curvature with one conical point. Canad. Math. Bull. 62 (2019), 341--347, arXiv:1712.04588
 
 \bibitem{Kim} Y.-H. Kim, Surfaces with boundary: their uniformizations, determinants of Laplacians, and isospectrality. Duke Math. J.  144 (2008), 73--107, arXiv:math/0609085   
 

\bibitem{Klevtsov} S. Klevtsov, Lowest Landau level on a cone and zeta determinants, J.Phys. A: Math. Theor. 50 (2017), 234003

\bibitem{Kraus2011} D. Kraus, O. Roth, T. Sugawa, Metrics with conical singularities on the sphere and sharp extensions of the theorems of Landau and Shottky. Math. Z. 267 (2011), 851--868


\bibitem{KRV-DOZZ} A. Kupiainen, R. Rhodes, V. Vargas, Integrability of Liouville theory: proof of the DOZZ Formula,  Ann. of Math.   191(2020),  81--166,	arXiv:1707.08785

\bibitem{LT} F. Luo, G. Tian, Liouville equation and spherical convex polytopes,
Proc. Amer. Math. Soc. 116 (1992), 1119--1129

\bibitem{Matsumoto} K. Matsumoto, Asymptotic expansions of double zeta functions of Barnes, of Shinttani, and Eisenstein series, Nagoya Math. J. 172 (2003), 59--102



\bibitem{OPS} B. Osgood, R. Phillips, P. Sarnak, Extremals of Determinants of Laplacians. J. Funct. Anal. 80 (1988), 148--211

\bibitem{OPS.5} B. Osgood, R. Phillips, P. Sarnak, Compact isospectral sets of surfaces. J. Funct. Anal. 80 (1988), 212--234

 \bibitem{OPS1} B. Osgood, R. Phillips, P. Sarnak, Moduli space, heights and isospectral sets of plane domains, Ann. of Math. 129 (1989), 293--362
 
\bibitem{Picard} E. Picard, De l'int\'egration de l'\'equation $\Delta u=e^u$ sur une surface de Riemann ferm\'ee, J. Reine Angew. Math. 130 (1905), 243--258

\bibitem{Polch} J. Polchinski, Evaluation of the one-loop string path integral, Commun. Math. Phys. 104 (1986), 37--47

\bibitem{Pol} A. Polyakov, Quantum geometry of Bosonic strings, Phys. Lett. B 103 (1981), 207--210

\bibitem{Sarnak}  P. Sarnak, Determinants of Laplacians; heights and finiteness. Analysis, et cetera, pp. 601--622. Academic Press, Boston, MA (1990)


\bibitem{Spreafico} M. Spreafico, Zeta function and regularized determinant on a  disk and on a cone,  Journal of Geometry and Physics 54 (2005), 355--371
\bibitem{SpreaficoZerbini}  M. Spreafico, S. Zerbini, Spectral analysis and zeta determinant on the deformed spheres,  Commun. Math. Phys. 273 (2007), 677--704
\bibitem{Spreafico2} M. Spreafico, On the Barnes double zeta and Gamma functions,  J. Number Theory 129 (2009), 2035--63


\bibitem{T-Z} L. Takhtajan, P. Zograf, Hyperbolic 2-spheres with conical singularities, accessory parameters and K\"ahler metrics on $\mathcal M_{0,n}$, Trans. Amer. Math. Soc. 355 (2003), no. 5, 1857--1867

\bibitem{T-Z2019}  L.A. Takhtajan, P.G. Zograf,  Local index theorem for orbifold Riemann surfaces. Lett. Math. Phys. 109 (2019), pp. 1119--1143, arXiv:1701.00771

 \bibitem{TroyanovSSC} M. Troyanov, Les surfaces euclidiennes \`a singularit\'es coniques, L'Enseignement Math\'ematique 32 (1986), 79--94.

\bibitem{Troyanov} M. Troyanov, Prescribing curvature on compact surfaces with conical singularities, Trans. Amer. Math. Soc. 134 (1991), 792--821 

\bibitem{TroyanovSp}   M. Troyanov, Metrics of constant curvature on a sphere with two conical singularities, Lecture Notes in Math. 1410 (1989), 296--306 
\bibitem{UY} M. Umehara, K. Yamada, Metrics of constant curvature $1$ with three conical singularities on the $2$-sphere. Illinois J. Math. 44 (2000),  72--94
\bibitem{Weisberger}  W. Weisberger, Conformal invariants for determinants of Laplacians on Riemann surfaces, Commun. Math. Phys. 112 (1987), 633--638

 \bibitem{Went}  R. Wentworth, Precise constants in bosonization formulas on Riemann surfaces,  Commun. Math. Phys. 282 (2008), 339--355


\bibitem{Z-Z} A. Zamolodchikov, Al. Zamolodchikov, Conformal bootstrap in Liouville field theory, Nuclear Physics B 477 (1996), 577--605



\end{thebibliography}
\end{document}